\documentclass{amsart}

\usepackage[utf8]{inputenc}
\usepackage{amsfonts,amssymb,amsmath,amsthm, mathtools}
\usepackage{graphicx,pgf,verbatim}
\usepackage[normalem]{ulem}
\usepackage{enumitem}

\usepackage{accents}
\usepackage{tikz}
\usetikzlibrary{arrows}
\usepackage{courier}
\usepackage{type1cm}
\usepackage{url}
\usepackage{subcaption}
\usepackage{multirow}
\usepackage[backend=biber, style=numeric]{biblatex}
\usepackage{interval}
\usepackage{hyperref}
\hypersetup{
     colorlinks,
     linkcolor={red!50!black},
     citecolor={green!50!black},
     urlcolor={blue!80!black}
   }
\usepackage{algorithm}
\usepackage{algpseudocode}

\bibliography{gmnp}

\numberwithin{equation}{section}

\begin{document}

\title[Tropical Laurent series]{Tropical Laurent series, their tropical roots, and localization results for the eigenvalues of nonlinear matrix functions}


\author{Gian Maria Negri Porzio}
\address{School of Matehmatics, University of Manchester, Alan Turing Building,
Oxford Rd, 
Manchester, UK}
\email{gianmaria.negriporzio@manchester.ac.uk}

\author{Vanni Noferini}
\address{Aalto University, Department of Mathematics and Sytems Analysis, 
P.O. Box 11100, FI-00076, Aalto, Finland.}
\email{vanni.noferini@aalto.fi}
\thanks{Vanni Noferini was supported by an Academy of Finland grant (Suomen Akatemian p\"{a}\"{a}t\"{o}s 331240).}

\author{Leonardo Robol}
\address{Department of Mathematics, University of Pisa, Largo Bruno Pontecorvo, 5, 56127 Pisa (PI), Italy.}
\email{leonardo.robol@unipi.it}
\thanks{Leonardo Robol was supported by the INdAM/GNCS research project 
``Metodi low-rank per problemi di algebra lineare con
struttura data-sparse''.}

\subjclass[2020]{15A80, 15A18, 47A10, 47A56}

\date{}

\dedicatory{}

\newcommand{\C}{\mathbb C}
\newcommand{\R}{\mathbb R}
\newcommand{\N}{\mathbb{N}}
\newcommand{\Z}{\mathbb{Z}}
\newcommand{\iu}{\ensuremath{\mathrm{i}}} 
\newcommand{\eu}{\ensuremath{\mathrm{e}}} 
\newcommand{\nbyn}{\ensuremath{n \times n}}
\newcommand{\Cnn}{\ensuremath{\C^{\nbyn}}}
\newcommand{\arxiv}[1]{\hreff{http://www.arXiv.org/abs/#1}{arXiv:#1}}
\newcommand{\cond}{\operatorname{\kappa}}
\newcommand{\wt}{\widetilde}
\newcommand{\ol}{\overline}
\newcommand{\bigO}{\mathcal{O}}
\newcommand{\tset}{\Omega}

\DeclarePairedDelimiter\abs{\lvert}{\rvert}
\DeclarePairedDelimiter\norm{\lVert}{\rVert}
\DeclarePairedDelimiter\floor{\lfloor}{\rfloor}

\newtheorem{theorem}{Theorem}[section]
\newtheorem{lemma}[theorem]{Lemma}
\newtheorem{proposition}[theorem]{Proposition}
\newtheorem{remark}[theorem]{Remark}
\newtheorem{example}[theorem]{Example}
\newtheorem{corollary}[theorem]{Corollary}
\newtheorem{definition}[theorem]{Definition}

\newcommand{\rmax}{\R_{\max}}
\newcommand{\rmaxm}{\R_{\max,\times}}
\newcommand{\trop}{\operatorname{\mathsf{t\!}}}
\newcommand{\tropx}{\operatorname{\mathsf{t\!}_\times\!}}
\newcommand{\newti}[1]{\mathcal{K}_{#1}}
\newcommand{\newt}[1]{\mathcal{N}_{#1}}
\newcommand{\ann}{\operatorname{\mathcal{A}}}
\newcommand{\annop}{\operatorname{\accentset{\hspace{4.5pt}\circ}{\ann}}}
\newcommand{\rr}{\ensuremath{R_{1}}}  
\newcommand{\RR}{\ensuremath{R_{2}}} 
\newcommand{\Ft}{\ensuremath{\wt{F}}}
\newcommand{\Bt}{\ensuremath{\wt{B}}}
\newcommand{\cj}{\ensuremath{\cond{(B_{k_{j}})}}}
\newcommand{\cs}{\ensuremath{\cond{(B_{k_{s}})}}}
\newcommand{\disk}{\operatorname{\mathcal{D}}}
\newcommand{\diskop}{\accentset\circ{\disk}}
\newcommand{\tp}{\oplus}
\renewcommand{\tt}{\otimes}
\renewcommand{\epsilon}{\varepsilon}

\newcommand{\normt}[1]{\ensuremath{\norm{#1}_{2}}}
\def\mystrut#1{\rule{0cm}{#1}}  

\begin{abstract}
Tropical roots of tropical polynomials have been previously studied
and used to localize roots of classical polynomials and eigenvalues
of matrix polynomials. We extend the theory of tropical roots from
tropical polynomials to tropical Laurent series. Our proposed
definition ensures that, as in the polynomial case, there is a
bijection between tropical roots and slopes of the Newton polygon
associated with the tropical Laurent series. We show that, unlike in
the polynomial case, there may be infinitely many tropical roots;
moreover, there can be at most two tropical roots of infinite
multiplicity. We then apply the new theory by relating the inner and
outer radii of convergence of a classical Laurent series to the
behavior of the sequence of tropical roots of its
tropicalization. Finally, as a second application, we discuss
localization results both for roots of scalar functions that admit a
local Laurent series expansion and for nonlinear eigenvalues of
regular matrix valued functions that admit a local Laurent series
expansion.
\end{abstract}

\maketitle

\section{Introduction}

Tropical algebra is a branch of mathematics developed in the
second half of the twentieth century. The adjective \emph{tropical} was
coined in honour of the Brazilian mathematician Imre Simon, who in 1978
introduced a min-plus structure on
$\N\cup \{\infty\}$ to be used in automata theory~\cite{si:78}. In the
following decades, researchers started using the nomenclature we are
now familiar with. For instance, Cuninghame-Green and Meijer
introduced the term \emph{max-algebra} in 1980~\cite{cume:80}, while
\emph{max-plus} appeared in the nineties and early
two-thousands~\cite{baco:92, mc:06}.

Even though tropical algebra has found many other applications,
in computational mathematics a preeminent one has been to approximately and cheaply locate the roots of
polynomials or, more in general, the eigenvalues of matrix
polynomials.  Indeed, having a cheap initial
estimate of these values is often very important for numerical
methods: for instance, contour integral algorithms  heavily rely on
knowing inclusion region for the eigenvalues~\cite{assa:09, assa:10,
  be:12}; in 2009, Betcke proposed a tropical algebra based diagonal scaling for a 
matrix polynomial $P$ to improve the conditioning of its eigenvalues
$\lambda_{j}$ near a target eigenvalue $\lambda$, which
requires the knowledge of $\abs{\lambda}$~\cite{be:08}; the
Ehrlich--Aberth method for the polynomial eigenvalue problem needs a starting point~\cite{bino:13} which can also be estimated via tropical algebra.  In 1996, Bini proposed an
algorithm to compute polynomial roots based on Aberth's method and
the Newton polygon, but without the formalism of tropical
algebra~\cite{bi:96}.  In 2009, Gaubert and Sharify showed that the
tropical roots of a scalar polynomial can give bounds on its
``standard'' roots~\cite{gash:09}. Later, Noferini, Sharify and
Tisseur~\cite{nosh:15} proposed similar results for matrix-valued
polynomials, and compared their bounds with the ones
in~\cite{binosh:13, me:13}.

The goal of the present paper is to generalize those results to the
case of meromorphic (matrix-valued) functions expressed as Laurent
series. The article is structured as follows: 
Section~\ref{sec:tropical} is devoted to extending the notion of
tropical polynomials to tropical Laurent series and to defining the
concept of tropical roots in these settings. In
Section~\ref{sec:evsLoc} we consider how tropical roots can give
bounds on the eigenvalues of meromorphic matrix-valued functions, just
like in the polynomial case.  Finally, in Section~\ref{sec:applications}
we show how to update the infinite Newton polygon of a tropical
Laurent function under specific circumstances, and how the results of
the previous section may benefit future contour integral eigensolvers.



\section{Algebra and analysis of tropical Laurent series}
\label{sec:tropical}
\subsection{Tropical algebra}
The \emph{max-plus semiring} $\rmax$ is the algebraic structure
composed of the set $\R \cup \{-\infty\}$ equipped with the max
operation $\tp$ as addition, and the standard addition $\tt$ as
multiplication, respectively. The zero and unit elements of this
semiring are $-\infty$ and $0$.

Similarly, one defines the \emph{max-times semiring} $\rmaxm$ as the set of
nonnegative real numbers $\R^{+}$ equipped with the max operation
$\tp$ as addition and the standard multiplication $\tt$ as
multiplication, respectively. Under these settings, the zero and unit
element are $0$ and $1$. $\rmaxm$ is isomorphic with $\rmax$ thanks to
the map $x \mapsto \log x$, with the usual convention $\log(0) =-\infty$.
We refer with the word \emph{tropical} to any of these two structures, as
many other researchers have done before us.

\subsection{From tropical polynomials to tropical series} 
\label{sec:tropicalseries}

A max-plus \emph{tropical polynomial} $\trop p$ is a function
of a variable $x\in\rmax$ of the form
\begin{equation*}
\trop p(x) :=\bigoplus_{j=0}^d \log b_j\otimes x^{\otimes j}
            =\max_{0\leq j\leq d} (\log b_j +j x),
\end{equation*}
where $d$ is a nonnegative integer and $b_0,\dots,b_d\in\R^{+}$. The
\emph{degree} of $\trop p(x)$ is $d$ if $b_{d} \neq 0$. The
finite tropical roots $\log \alpha$ are the points where the maximum is attained at
least twice, which correspond 
to the non-differentiable points of the function $\trop p(x)$.
For such roots, 
the multiplicity can be defined as the jump of the derivative at
$\log\alpha$, that is 
$\lim_{\epsilon\rightarrow 0}\frac{d \trop
	p}{dx}\vert_{\log \alpha+\epsilon}-\frac{d \trop p}{dx}\vert_{\log \alpha-\epsilon}.  $
The multiplicity of $-\infty$ as a root of $\trop p$ is given by
$\inf\{j\mid b_j \neq 0\}$.

 Similarly, a max-times tropical polynomial is the function
\begin{equation*}
\tropx p(x) : = \bigoplus_{j=0}^d  b_j\otimes x^{\otimes j}
            =\max_{0\leq j\leq d} (b_j x^j),
\end{equation*}
with $x, b_{0}, \dots, b_d \in \R^{+}$ and again the tropical roots $\alpha$
are the points of non-differentiability of the function above; this
time the multiplicity is the jump of the derivative of the logarithm
of the polynomial at the root. The tropical roots of a max-times
tropical polynomial with coefficients $b_0,\dots,b_d$ are the
exponentials of the tropical roots of the max-plus tropical
polynomials with coefficients $\log b_0,\dots,\log b_d$, including
by convention $-\infty=:\log 0$, and have the same multiplicities:
this follows immediately if one observes that for any
$x \in \R^+$ it holds
\[ \max_{0\leq j\leq d} (\log b_j +j \log x) = \log \max_{0\leq j\leq
    d} (b_j \cdot x^j). \]

A tropical version of the fundamental theorem of algebra holds,
implying that a tropical polynomial of degree $d$ has $d$ tropical
roots counting multiplicities~\cite{cume:80}. Using the Newton polygon, the computation of the
tropical roots can be recast as a geometric problem which amounts to computing the upper convex hull of the
points $(j, \log b_j)$ (see, for example, the introduction of
\cite{sh:11}). This can be done in $\mathcal{O}(d)$ operations
using the Graham scan algorithm, since the points are already sorted
by their abscissa (see Section~\ref{sec:newton-polygon})~\cite{gr:72}.

We are concerned with a generalization of tropical polynomials 
that can be helpful in studying analytic functions, similarly to
the role played by tropical polynomials in locating the roots (or 
eigenvalues) of (possibly matrix-valued) functions.

\begin{definition}[Tropical Laurent series]
  \label{def:tropicalLaurentseries}
  Let $(b_j)_{j \in \Z}$ be a sequence of elements of $\R^+$, indexed
  by integers and not all zero. The associated max-times tropical Laurent series
  is
        \[
          \tropx f(x) :=  \sup_{j \in \mathbb Z} (b_j x^j), 
	\]
        with $x \in \R^+$ and it takes values in $\R^+\cup
        \{+\infty\}$. Similarly, the associated max-plus tropical Laurent series
        is a function of variable $x \in \R \cup \{-\infty\}$
        \[
	  \trop f(x) := \sup_{j \in \mathbb Z} (\log b_j + jx), 
	\]
	taking values in $\R \cup \{\pm \infty \}$. In addition, let
        $f(\lambda) = \sum_{j\in\Z}b_{j}\lambda^{j}$ be a complex
        function defined by a Laurent series. Then
        $\tropx f(x) = \sup_{j \in \mathbb Z} (\abs{b_j} x^j)$ is the
        \emph{max-times tropicalization} of $f(\lambda)$, while
        $ \trop f(x) := \sup_{j \in \mathbb Z} (\log |b_j| + jx)$ is
        its \emph{max-plus tropicalization}.
\end{definition}

\begin{remark}
  There are two obvious differences in Definition \ref{def:tropicalLaurentseries}
  with respect to the polynomial case. The first is the use of supremum in lieu
  of maximum; this is necessary because there might be points $x$
  where the value of $\tropx f(x)$ is not attained by any of the
  polynomial function $b_jx^j$.  In addition, note that $\tropx
  f(x)$ ($\trop f(x)$, respectively) may not be well-defined on
  the whole $\rmaxm$ ($\rmax$, respectively). The next sections are
  going to clarify these aspects.
\end{remark}

The next lemma is a well-known result of convex analysis~\cite[Theorem~5.5]{ro:70}. We decided
to recall it to make this section self-contained.

\begin{lemma}\label{lem:convex}
  Let $g_j(x)$ be a set of real-valued functions all defined and
  convex on the same interval $\Omega \subseteq \R$, and indexed over
  some non-empty set $I$, possibly infinite or even uncountable. Then,
  the largest domain of definition of the function
\[ g : D \rightarrow \mathbb R, \qquad x \mapsto g(x) = \sup_{j \in I} g_j(x) \]
is an interval $D \subseteq \Omega \subseteq \R$; moreover, $g(x)$ is convex on $D$.
\end{lemma}


\begin{definition} \label{def:I}
  Given a sequence of elements in $\mathbb R^+$ 
  $\{ b_j \}_{\mathbb Z}$, or equivalently a Laurent series 
  with non-negative coefficients $\sum_{j \in \mathbb Z} b_j x^j$, 
  we denote by $I:=\{ j \in \Z:
  b_j > 0 \}$ the set of indices corresponding to the nonzero 
  elements. 
\end{definition}

\begin{proposition}\label{prop:convex}
Let $(b_j)_{j \in \Z}$ be a sequence of elements of $\R^+$,
indexed by integers and not all zero. Let $I$ be its nonzero 
elements as in Definition~\ref{def:I}.
 Then the following statements are true:
\begin{enumerate}
\item The domain of the function
  \[
    \tropx f: D \rightarrow \R^+, \qquad x \mapsto
    \tropx f(x) = \sup_{j \in I} b_j  x^j
  \] is an interval $D
\subseteq \R^+$; moreover, $\tropx f(x)$ is convex on $D$.

\item The domain of the function $\trop f$, restricted to $\mathbb R$, 
  \[
    \trop f_{|\mathbb R}: D_{+} \rightarrow \R, \qquad x \mapsto \trop f(x) = \sup_{j
      \in I} (\log b_j + j x)
  \] is an interval $D_{+} \subseteq \R$; moreover, 
  $\trop f_{|\mathbb R}(x)$ is convex on $D_{+}$.
\end{enumerate}
\end{proposition}
\begin{proof}
This easily follows from Lemma~\ref{lem:convex}. Indeed:
\begin{enumerate}
\item For any $j \in I$, $b_j x^j$ is convex on $\R^+$ (note $j \in I \Rightarrow b_j > 0)$.
\item For any $j \in I$, $\log b_j \in \R$. Hence, $\log b_j + j x$ is
  convex, on $\R$.
\end{enumerate}
\end{proof}

Proposition \ref{prop:convex} does not specify the nature of the interval
$D$, which can be open, closed, or semiopen (either side being
open/closed): indeed, examples can be constructed for all four
cases. It can also happen that $D = \R^+$, that $D$ is empty, or that
$D$ is a single point. In addition, if $\tropx f(x)$ and $\trop f(x)$ are, resp., the max-times and the max-plus tropicalization of the same (classical) function $f(x)$, then $D = \interval{a}{b}$ is the
domain of $\tropx f(x)$ if and only if $D_+ := \interval{\log a}{\log b}$ is
the domain of $\trop f(x)$. (Similar statements hold when $D$ is not closed but open or semiopen.)

An immediate consequence of Proposition~\ref{prop:convex} is that
\[ 
  Y:=\{ y \in D^{\mathrm{o}} : \tropx f(x) \text{ is not
    differentiable at } x=y \} \Rightarrow \# Y \leq \aleph_0. 
\]
where $D^{\mathrm{o}}$ is the interior of $D$. In
particular, $\tropx f(x)$ (as any convex function of a real variable)
has left and right derivative everywhere in its domain of definition
$D$ and it is differentiable almost everywhere in the interior of $D$,
except for at most countably many points at which the left and right
derivative differ.  On the other hand, if $\# I \geq 2$ then it is
clear that $Y$ may be not empty. Indeed, we will define the set of
tropical roots of $\tropx f(x)$ as
$Y \cup (D \setminus D^{\mathrm{o}})$, together (possibly) with $\{ 0 \}$. Via the logarithmic map, an analogous definition holds for the tropical roots of 
$\trop f(x)$.

\begin{definition}[Tropical roots of Laurent series]
\label{def:tropicalroot}
Consider a max-plus Laurent series
	\[
	  \trop f_{| \mathbb R}(x) := \sup_{j \in \Z} (\log b_j + jx), 
	\]
as a real valued function defined on some interval $D_{+}$.
A point $\log \alpha \in D_{+}$ 
is said to be a \emph{tropical root} of $\trop f(x)$ if:
\begin{enumerate}

\item either $\trop f_{| \mathbb R}(x)$ is not differentiable at $\log \alpha\in \R$, and in
  this case the multiplicity of $\log \alpha$ is the size of the jump of
  the derivative at $\log \alpha$, i.e.,  
	$m := \frac{d}{dx^+} \trop f(\log \alpha) - \frac{d}{dx^-}
        \trop f(\log \alpha)$;

      \item or $\log \alpha = -\infty$ if $b_{j} = 0$ for $j\leq 0$
        and in this case the
  multiplicity is given by $\inf\{j\mid b_j \neq 0 \}$;

\item or $\log \alpha \in \R$ is a \emph{finite} endpoint of $D_{+}$, and in this case the
  multiplicity of $\log \alpha$ is $m=\infty$. 
\end{enumerate}
Similarly  consider $  \tropx f(x) :=  \sup_{j \in \Z} (b_j x^j)$ 
as a real valued function defined on some interval $D$. Then, $\alpha
\in D$ is a tropical root of $\tropx f(x)$ with multiplicity $m$ if
$\log \alpha$ is a tropical root of $\sup_{j \in \Z} (\log b_j + j x)$
with multiplicity $m$. 
\end{definition}



As expected, Definition~\ref{def:tropicalroot} falls back to
the polynomial case when $\tropx f(x)$ is a polynomial: if
$\tropx f(x)$ is a polynomial, then
$D = \interval[open right]{0}{\infty}$, hence the second subcase falls
back to $0$ being a root of $\tropx f(x)$, while the third one never
happens. In addition note that the the multiplicity of a tropical root
may be infinite. This happens if $\alpha$ is nonzero, is an endpoint of $D$, and
belongs to $D$, that is, if $\tropx f(\alpha) \in \R$ but either
$\tropx f(x) = +\infty$ for $x>\alpha$ or $\tropx f(x)=+\infty$ for
$x<\alpha$. Analogous observations can be made for $\trop f(x)$.

Before continuing to expose our theory, we give some examples to gain further insight on the different cases in the definition of tropical roots.



\begin{example} \label{ex:1}
Consider $f(\lambda) = -\log(1-\lambda)$ and its Taylor expansion in $[-1,1[$. 
Then
\[
  f(\lambda)= \sum_{j=1}^{\infty}\frac{\lambda^j}{j} \Rightarrow \tropx
  f(x)=\sup_{j \geq 1} \left(\frac{x^j}{j}\right) =
  \begin{cases} 
x \ &\mathrm{if} \ x \leq 1;\\
\infty \ &\mathrm{if} \ x > 1.
\end{cases}
\]
The domain of $\tropx f(x)$ is $D=\interval{0}{1}$. The right endpoint
$\alpha_1 = 1 \in D$ is a tropical root of infinite multiplicity. The
point $\alpha_{0} = 0$ is a simple root, because $b_j = 0$ for $j\leq
0$.  They are the only tropical
roots because $\tropx f(x)$ is differentiable everywhere else in $D$.
If we had considered $g(\lambda) = 1-\log(1-\lambda)$, then $\tropx g(x)$ would not
have had $\alpha_{0} = 0$ as a tropical root, even though
$D=\interval{0}{1}$: in fact, the special condition for $0$ to be a tropical root would not be satisfied, as for an endpoint $a$ to
be a root $\log a$ needs to be finite.
\end{example}

\begin{example}
  \label{ex:2}
Let $H_j = \sum_{k=1}^jk^{-1}$ denote the $j$-th harmonic number and consider 
\begin{equation}
  \label{eq:ex2}
  f(\lambda) = \sum_{j=1}^\infty \eu^{H_j} \lambda^j \Rightarrow \tropx f(x) =
  \sup_{j \geq 1} (\eu^{H_j}x^j).
\end{equation}
The domain of $\tropx f(x)$ is $D=\interval[open right]{0}{1}$. As in
the previous example, $\alpha_{0} = 0$ is a root with multiplicity $1$
because $b_j = 0$ for $j \leq 0$. The
points of nondifferentiability are
\[
  \alpha_j  = \eu^{-1/j}, \qquad j=1,2,3,\dots
\]
which  have all multiplicity $1$, and accumulate at $1$. In
Figure~\ref{fig:harmon} we plotted $\tropx f(x)$ and $\alpha_{j}$. Note that
$\tropx f(x) = + \infty$ if and only if $x \geq 1$. However, $1
\not \in D$, so $1$ itself is not a tropical root.
\begin{figure}
  \centering
  \begin{subfigure}{0.49\linewidth}
    \centering\includegraphics[width=0.95\textwidth]{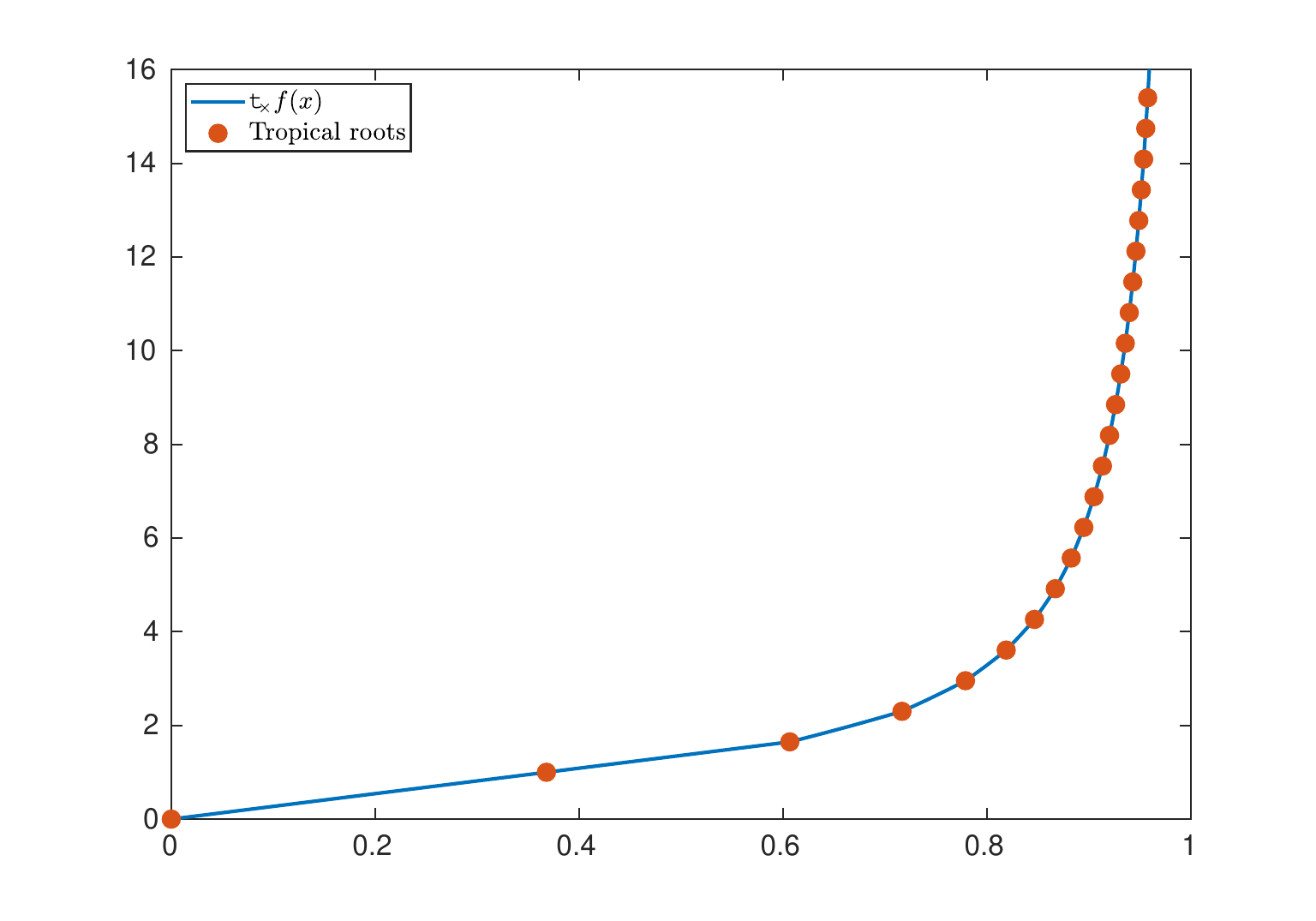}
    \caption{}
    \label{fig:harmona}
  \end{subfigure}~
  \begin{subfigure}{0.49\linewidth}
    \centering\includegraphics[width=0.95\textwidth]{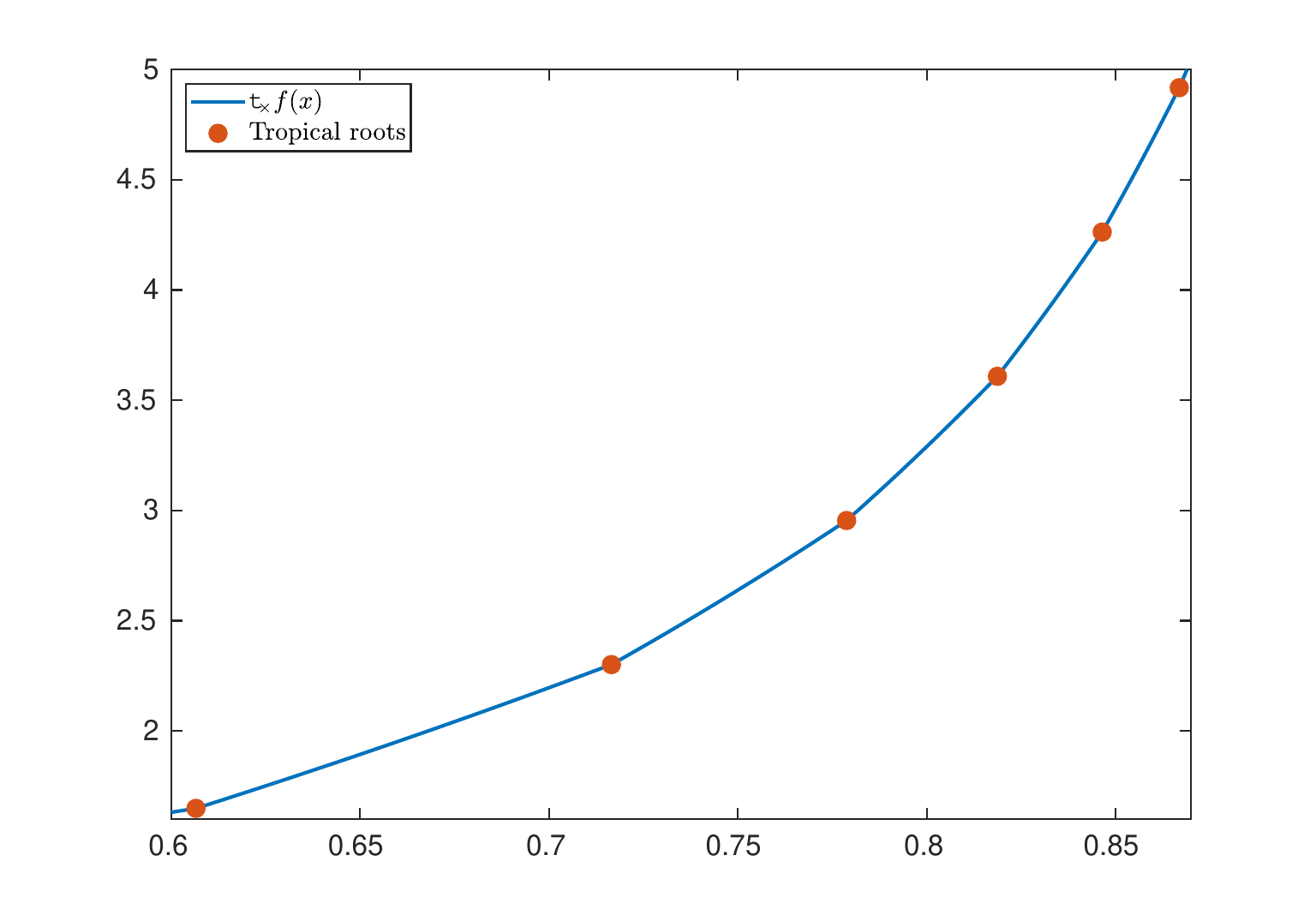}
    \caption{}
    \label{fig:harmonZoom}
  \end{subfigure}
  \caption{The  plot of $\tropx f(x)$ defined in~\eqref{eq:ex2} and its
    tropical roots $\alpha_{j}$ in~\ref{fig:harmona}, and a zoom
    in~\ref{fig:harmonZoom}, where we can  see that the nonzero
    $\alpha_{j}$ are the points of nondifferentiability.}
\label{fig:harmon}
\end{figure}

\end{example}

\begin{example}
Let
\[
 f(\lambda) = \sum_{j=0}^\infty \eu^{j^2} \lambda^j \Rightarrow \tropx f(x) =
 \sup_{j \geq 0} (\eu^{j^2}x^j) \equiv + \infty.
\] 
In this case, the domain of $\tropx f(x)$ as a real function is empty,
and hence there are no tropical roots.
\end{example}

We now report an example of a tropical series 
$\tropx f(x)$ that has tropical roots accumulating on the 
right boundary, and also a largest tropical root. 

\begin{example} \label{ex:infty}
  Let $b_j := \prod_{k = 1}^j (e^{2^{-k} - 1})$, and 
  \[
    f(\lambda) = \sum_{j = 0}^{\infty} {b_j} \lambda^j
    \Rightarrow
    \tropx f(x) = \sup_{j \geq 0} b_j x^j.
  \]
  The domain of $\tropx f(x)$ is $D = [0, e]$. The right endpoint 
  $e$ is therefore a tropical root (while $0$ is not); 
  by a direct computation we see that the points of 
  non-differentiability are $\alpha_j = \exp(1 - 2^{-j})$ for 
  $j = 1, 2, 3, \ldots$. The $\alpha_j$ have $e$ as accumulation point.
\end{example}

\begin{remark}
From now on we will assume that $0$ is never a tropical root
of $\tropx f(x)$. On one hand, this is not a restriction, because if
$0$ is a root of multiplicity $m$, then it means $b_{j} = 0$ for $j
< m$ and thus we can consider a shifted version of $\tropx
f(x)$. On the other, it greatly simplifies the exposition, because in
this case $\alpha$ is a tropical root if it is a nondifferentiable
point or a closed endpoint of $D$.
\end{remark}

As a practical example, the function of Example~\ref{ex:1} can be modified 
to remove the zero root dividing it by $\lambda$, and thus yielding
\[
  \hat f(\lambda) = \sum_{j = 0}^{\infty} \frac{\lambda^j}{j + 1}. 
\]
For this modified function, the domain of $\tropx \hat f(x)$ is still $D = [0, 1]$. However, 
in this case there is no longer a tropical root at zero, but the only tropical 
root is $1$, still with infinite multiplicity. 

Before discussing the connection between tropical Laurent series
$\tropx f(x)$ and the roots of the related function $f(\lambda)$, we shall
see how the relation between tropical roots and the associated Newton
polygon extends from the polynomial setting to the Laurent series
case. In addition, in Example~\ref{ex:2} and \ref{ex:infty} 
we have witnessed that the
non differentiability points converge to the right
endpoint of $D$. 
This does not happen by chance, and it
will be the second topic of the next section.

Since once again it is clear that one can map a max-times Laurent
series to a max-plus Laurent series by taking the logarithm of the
coefficients (and the roots are also mapped under the logarithm), we
will for convenience focus on max-times Laurent series to analyse the
connection with the Newton polygon.

\subsection{Asymptotic behaviour of tropical roots and infinite Newton
  polygons}
\label{sec:newton-polygon}

The goals of this subsection are
multiple. First, we show that all the tropical roots are
isolated, except for two possible accumulation points. Then we prove a
result that mirrors Definition~\ref{def:tropicalroot}:
$\tropx f(x) \in \R^{+}$ if
$x\in \interval[open]{\alpha_{-\infty}}{\alpha_{+\infty}}$, where $\alpha_{\pm\infty}$ will be defined in Definition \ref{def:ainfty}.  Finally, we define
the Newton polygon, which can be used to compute the tropical roots
with their multiplicities. Even though it is a standard tool for the
roots of tropical polynomials, it requires few more technicalities in
the Laurent case: some properties are less obvious because we are
dealing with a  possibly infinite, albeit countable, number of roots.  


We begin with a preliminary lemma: even though the tropical roots may
be countably infinite, we show that they are isolated, except possibly
for two accumulation points.

\begin{lemma}
\label{lem:isolated}
  Let $\tropx f(x)$ be a max-times Laurent series with 
  non-differentiable points 
  $(\alpha_{k})_{k\in T'}$ indexed by $T' \subseteq \mathbb Z$. 
  Then all the 
  $\alpha_{k}$ are isolated with the only possible accumulation points
  being $\sup_{k \in T'} \alpha_k$ and $\inf_{k \in T'} \alpha_k$. 
  In particular, without loss of generality we can
  assume that the $\alpha_k$ are sorted in a 
  non-decreasing way, i.e., that if $j > k$ then 
  $\alpha_j > \alpha_k$. 
\end{lemma}
\begin{proof}
 Observe that we can write:
\begin{equation}
\label{eq:isolated}
	  \tropx f(x) := \sup_{j \geq 0} g_j(x), \qquad 
	  g_j(x) := \max_{\substack{k \in \mathbb Z \\
                \abs{k} \leq j}} (b_kx^k).
\end{equation}
        In other words, $\tropx f(x)$ can be approximated from below by the
functions $g_j(x)$. 
Therefore we
consider the non-differentiable points of $g_j(x)$ and see how they
change as we let $j\to \infty$.

First, note that $g_0(x)$ is the constant function $g_0(x) \equiv b_0$. In
general, the function $g_{j}(x)$ is obtained by $g_{j-1}(x)$ adding at most two
more monomial functions to the set whose supremum is taken; more precisely, one
function is added if and only if $b_j\neq 0$ and the other if and only if
$b_{-j}\neq 0$. The functions to be added have a larger exponent (in
absolute value) than all the others considered before. We claim that this automatically
introduces at most two new non-differentiable points, which will become the
leftmost and the rightmost ones. Hence, there are at most two accumulation points,
that (if present) are indeed $\sup_{k\to\pm \infty}\alpha_k$
and $\inf_{k\to\pm \infty}\alpha_k$, proving the statement.

It remains to prove the claim. For the sake of simplicity, we consider only
what happens at the right hand side of our domain: the proof for the left hand
side is analogous. There are two cases: either the new non-differentiable point,
say $\alpha$, is already the rightmost one, or it superposes or lies on the left
of one or more previous non-differentiable points. In the first scenario, a new
non-differentiable point $\alpha$ is created; in the latter, every 
non-differentiable point larger
than $\alpha$ disappears, and $\alpha$ becomes the rightmost one. In
Figure~\ref{fig:isolated} we plotted these two cases while building $g_{2}(x)$:
in~\ref{fig:isolated1} the new point is already the rightmost one, while
in~\ref{fig:isolated2} it superposes with the previous point.
\end{proof}
\begin{figure}

  \centering
  \begin{subfigure}{0.49\linewidth}
    \centering\includegraphics[width=0.95\textwidth]{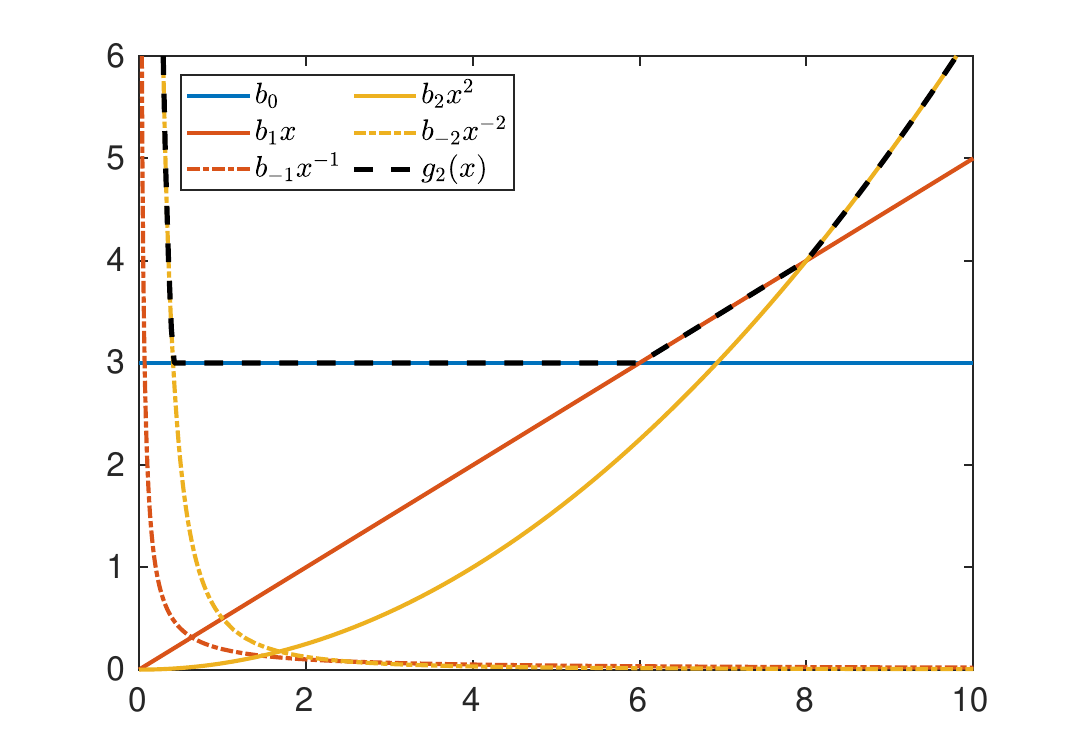}
    \caption{}
    \label{fig:isolated1}
  \end{subfigure}~
  \begin{subfigure}{0.49\linewidth}
    \centering\includegraphics[width=0.95\textwidth]{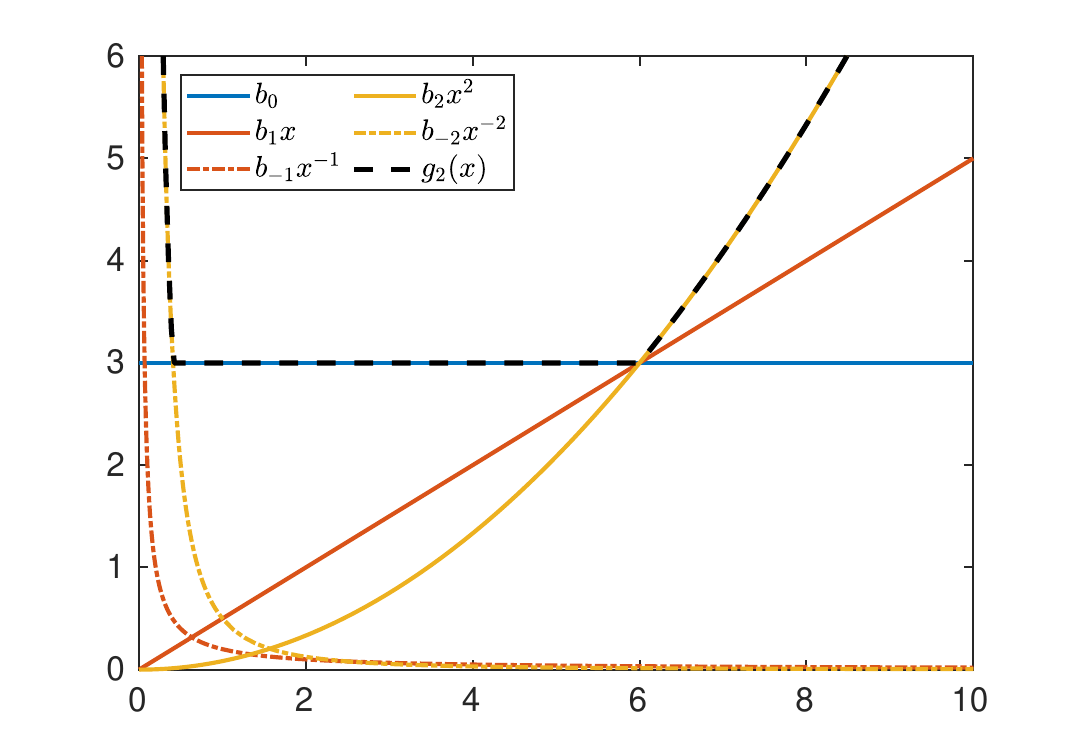}
    \caption{ }
    \label{fig:isolated2}
  \end{subfigure}
  \caption{Examples of possible scenarios for $g_{2}(x)$ (black dashed
    line). In~\ref{fig:isolated1} the new non differentiable point is the rightmost one,
    hence it becomes a distinct tropical root with multiplicity
    one. In~\ref{fig:isolated2} it superposes with a
      previous one, hence the multiplicity of the rightmost tropical
      root is larger than one.}
\label{fig:isolated}
\end{figure}
We have now proved that every tropical root $\alpha_{k}$ of $\tropx f(x)$
is either isolated or an endpoint of the domain.  
In addition, 
Lemma~\ref{lem:isolated}
guarantees that the roots can be indexed as $(\alpha_k)_{k \in T}$
with $T \subseteq \mathbb Z \cup \{ - \infty, +\infty \}$, and 
$T' \subseteq T$. Here, we make the choice that 
the indices $-\infty$ (resp $+\infty$) is used if and only if
the left (resp. right) endpoint  
is both accumulation point and a tropical roots. With this choice, 
we can assume, without loss of generality, that the roots
are sorted 
and the inherited topology from $\mathbb R$ 
is respected. The
nature of $T$ may vary considerably. It may be either bounded or unbounded
above, and either bounded or unbounded below, depending on whether the
sequence of tropical roots is bi-infinite, infinite only on the left,
infinite only on the right, or finite. For example, if $\tropx f(x)$ is a
polynomial, then $T$ is clearly finite. The next proposition proves
that if $\tropx f(x)$ has a finite number of roots, 
then the cardinality of $I$ (the set of indices corresponding 
to nonzero coefficients) is finite. 


\begin{proposition}
\label{prop:nRoots}
Let $\tropx f(x) = \sup_{j\in \Z} b_{j}x^{j}$  be a tropical Laurent
series. Let $I := \{j \in \Z : b_{j} > 0\}$ and let $(\alpha_k)_{k\in
  T}$ be tropical roots, with $T\subset \Z$. Then the set of indices
$I$ is bounded above if and only if $T$ bounded above and the domain 
$D$ of $\tropx f(x)$ is not
closed on the right. Similarly, $I$
is bounded below if and only if $T$ is bounded below and $D$ is not
closed on the left, or the left endpoint of $D$ is zero and $b_0 \neq 0$. 
\end{proposition}

\begin{proof}
  We first prove the case when $I$ is bounded above. 
  If $I$ is bounded above, i.e., $b_j = 0$ definitely
  for $j$ large enough, then $T$ is bounded above and
  $\tropx f(x)\in \R^{+}$ for every $x$ large enough, because
\[
  \tropx f(x)= \sup_{\substack{j\in I \\ j>0}}b_{j}x^{j} =
  \max_{\substack{j\in I\\ j>0}} b_{j}x^{j}.
\]

Conversely, assume that $T$ is bounded above and that $D$ is not
closed on the right, i.e., there is no rightmost tropical root with infinite
multiplicity.  Let $\alpha_{p}$ be the rightmost tropical root and let
$j_p\in I$ be the index where the value $\tropx f(\alpha_p)$ is
attained for the last time. Then,
by definition of tropical root, for every $ x > \alpha_p $ it holds
\[
  b_{j_{p}}x^{j_{p}} \geq b_{j}x^{j},
\]
for every $j >0$, which implies that $\tropx f(x) \in \R^{+}$ for $x$
large enough, and therefore $D$ needs to be unbounded 
on the right. Equivalently, we can rewrite the previous equation as
    \begin{equation}
   \label{eq:contra1}
    x^{j-j_p} \leq \frac{b_{j_{p}}}{b_{j}}.
  \end{equation}
  Now assume that $I$ is not bounded above. Then there exist infinite
  values of $j$ such that $b_{j} > 0$ and $j > j_p$. This yields a
  contradiction, because $\lim_{x \to \infty}x^{j-j_p} = \infty$,
  but~\eqref{eq:contra1} implies this value is bounded above by a
  constant for every $x \geq \alpha_p$.

  The other statement follows by a similar reasoning, with the only 
  exception of zero which needs to be treated separately. If zero is in 
  the domain, then $-\infty$ is a tropical root (and therefore $T$ is 
  not bounded below), unless $b_0 \neq 0$. In the latter case, 
  zero is not a root and all the proof follows by the same argument as 
  above. 
\end{proof}

We are now ready to extend the definition 
$\alpha_{\pm \infty}$, covering any possible  $T$.  Note that if $T$ 
is bounded above then
$\sup T \in T$ and if $T$ is bounded below then $\inf T \in T$.

\begin{definition}
\label{def:ainfty}
  Given a sequence $(\alpha_k)_{k\in T}$ of tropical roots
  we define $\alpha_{\pm \infty}$ as elements of
  $\R^+ \cup \{+ \infty\}$ as follows:
  \[
    \alpha_{+ \infty} = \begin{cases}
      \sup_{k \in T} \alpha_k, & \text{if } \sup(T) = +\infty, \\ 
      +\infty & \text{otherwise}; \\
    \end{cases} \quad  
    \alpha_{- \infty} = \begin{cases}
      \inf_{k \in T} \alpha_k, & \text{if } \inf(T) = -\infty, \\ 
      0 & \text{otherwise}. \\
    \end{cases}
\] 
\end{definition}


By definition, if $\tropx f(x)$ is well defined as a real-valued function
on $D=\interval{a}{b}$, and $a >0$, then $a$ and $b$ are tropical roots of
infinite multiplicity. A priori, $\tropx f(x)$ may also have infinite
roots and it may happen, for example, that $\lim_{k\to+\infty}\alpha_k = b$: indeed,
Lemma~\ref{lem:isolated} does not exclude this scenario. 

Before proving the main result of this section, we state the
definition of the Newton polygon for Laurent series and show that the
usual result linking the tropical roots with the slopes of the Newton
polygon holds true for the isolated ones. To keep the exposition concise, we consider max-times tropical series: it is clear how to adapt the definition to the max-plus case via the exponential map.

\begin{definition}[Newton polygon] \label{def:newtpolygon}
  Let $\tropx {f}(x) = \sup_{j \in \mathbb Z} b_{j}x^{j}$ be a max-times tropical
  Laurent series. We define the Newton polygon $\newt{\tropx{f}}$
  as the upper convex hull of $\{(j, \log b_{j}) \mid b_j \neq 0 \}_{j \in \mathbb Z}$ 
  and we denote with
  \[
\newti{\tropx f} :=\{j \mid (j, \log b_{j}) \in \newt{\tropx{f}}\}
  \]
  the vertices of the Newton polygon for $\tropx f(z)$.
\end{definition}

\begin{remark}
  Note that the term ``polygon'' is used with a slight abuse of notation 
  in Definition~\ref{def:newtpolygon}, since the Newton polygon could 
  comprise an infinite number of vertices (and become what's sometimes 
  known as apeirogon). Nevertheless, we prefer to use this nomenclature 
  for consistency with the polynomial case. 
\end{remark}


\begin{lemma}[Newton Polygon for isolated roots]
\label{lem:newtonIsolated}
  Let $\tropx f(x)$ be a max-times tropical Laurent series, and $\newt{\tropx f}$
  its Newton polygon. Then for each segment
  connecting two vertices $(j_k, \log b_{j_k})$ and
  $(j_{k+1}, \log b_{j_{k+1}})$ with $j_k \neq j_{k+1}$ we have a tropical
  root $\alpha_k$ corresponding to a point of non-differentiability,
  with finite multiplicity $m_k$, where:
	\[
	  \alpha_k =  \left(\frac{b_{j_{k}}}{  b_{j_{k+1}}}\right)^{\frac{1}{j_{k+1} - j_k}}, \qquad 
	  m_k := j_{k+1} - j_k. 
	\]
        Equivalently, the isolated roots of $\tropx f(x)$ are the
        exponential of the opposite of the slopes of the corresponding segments.
\end{lemma}
\begin{proof}
By Lemma~\ref{lem:isolated}, all the tropical roots are isolated.
Hence, if $\alpha_k$ is a non-differentiable point for 
$\tropx f(x)$, then there exists $\epsilon > 0$ such that
	\[
	  \tropx f(x) = \begin{cases}
	    b_{j_k} x^{j_k} & \alpha_k - \epsilon \leq x \leq \alpha_k \\
	    b_{j_{k+1}}x^{j_{k+1}} & \alpha_k \leq x \leq \alpha_k + \epsilon \\
	  \end{cases}
	\]
for some $j_k \leq j_{k+1}$. The indices are determined by checking 
which functions (locally) realize the supremum 
$\trop f(x) = \sup_{k \in \mathbb Z} \log b_j + jx$; by definition of upper convex hull, they thus correspond to the vertices of the Newton polygon.
Indeed, the 
non-differentiability at $x = \alpha_k$ 
happens if
and only if
the segment connecting $(j_k, \log b_{j_k})$ and 
$(j_{k+1}, \log b_{j_{k+1}})$ is an edge of the Newton 
Polygon; this is equivalent to imposing that all points 
of the form $(j, \log b_j)$ for $j \not \in \{ j_k, j_{k+1} \}$
fall below the line extending the segment:
\[
    \log(b_j) < \log b_{j_k} + \frac{j - j_k}{j_{k+1} - j_k} \left(
        \log b_{j_{k+1}} - \log b_{j_k}
    \right), \qquad 
    \forall j \not \in \{ j_k, j_{k+1} \}. 
\]
Using $\log \alpha_k = -(\log b_{j_{k+1}} - \log b_{j_k}) / (j_{k+1} - j_k)$
we have the equivalent statement 
\[
    \log b_j - \log b_{j_k} + (j - j_k) \log \alpha_k < 0 
    \iff 
    b_j \alpha_k^j < b_{j_k} \alpha_k^{j_k}, 
\]
which is precisely the sought condition.

By equating the two expressions at
$\alpha_k$ we obtain
\[
  \alpha_k =  \left(\frac{b_{j_{k}}}{  b_{j_{k+1}}}\right)^{\frac{1}{j_{k+1} - j_k}}, \qquad 
  m_k = j_{k+1} - j_k, 
\]
where $m_k$ is the jump of the derivative of the max-plus 
version of the tropical polynomial. 

 \qedhere

\end{proof}

\begin{theorem}
  \label{thm:asymptotics}
	Let $\tropx f(x) = \sup_{j \in \Z} (b_jx^j)$ be a max-times
        Laurent series, and $(\alpha_k)_{k\in T}$ be the sequence of tropical
        roots. Let $\alpha_{\pm \infty}$ be defined as in Definition~\ref{def:ainfty}
	Then, 
	\[
          x \in \interval[open]{\alpha_{-\infty}}{\alpha_{+\infty}} \Rightarrow \tropx
          f(x) \in \R^+
        \]
	and 
	\[
	 x \not \in \interval{\alpha_{-\infty}}{ \alpha_{+\infty}}
         \Rightarrow \tropx f(x) = \infty. 
       \]
       Furthermore, $\tropx f(\alpha_{+\infty}) \in \R^+$ if and only if
       $\alpha_{+\infty}$ is a tropical root. 
      Respectively, $\tropx f(\alpha_{-\infty})  \in \R^+$, if and only if
      $\alpha_{-\infty}$ is a tropical root. 
\end{theorem}
\begin{proof}
We only prove the claim for $\alpha_{+\infty}$, i.e., we argue that
$\tropx f(x)$ is finite if $x < \alpha_{+\infty}$ and large enough,
  and that $\tropx f(x)$ is infinite if $x > \alpha_{+ \infty}$; the
  claim for $\alpha_{- \infty}$ admits an identical proof, so we
  omit it. We analyse separately three possible cases.
\begin{enumerate}
\item \emph{There is a largest tropical root and it has finite
    multiplicity.} First,  note that $\alpha_{+\infty} = +\infty$,
   hence the second part of the thesis is vacuously true. The fact
   that $\tropx f(x) \in \R^+$ for $x$ large enough follows immediately
   from Proposition~\ref{prop:nRoots}, because $b_{j} = 0$ for $j > 0$
   large enough. Here we have $\tropx f(\alpha_{+\infty}) = +\infty$
   by convention.

\item\emph{There is a largest tropical root with infinite
    multiplicity} In this case the statement of the theorem coincides
  with Definition~\ref{def:tropicalroot}, because a root has infinite
  multiplicity if and only if it is a closed endpoint of $D$, where $\tropx
  f(x)$ is a well-defined function. Hence $\tropx f(\alpha_{+\infty})\in\R^+$.

\item \emph{There is  no largest tropical root.} There
  are two subcases.
  
\begin{itemize}
\item If $\alpha_{+\infty} = +\infty$, then the set of tropical roots
  is unbounded above. Let $x \in \R^+$ be sufficiently large, then this
  implies that there exists $k$ such that $\alpha_{k-1} \leq x \leq
  \alpha_{k}$, and hence
  	\[ 
0 < b_{j_{k-1}}x^{j_{k-1}} \leq \tropx f(x) \leq b_{j_k}x^{ j_k} < +
\infty,
\]
while the second implication is vacuously true. Similarly to the first
case, $\tropx f(\alpha_{+\infty}) = +\infty$ by convention.

\item  Assume now $\alpha_{+ \infty} \in \R^+$. If $x =
  \alpha_{+\infty}-\epsilon$ for $\epsilon > 0$ and not too large, then
  there exist two tropical roots $\alpha_{k-1}
  < x < \alpha_{k}$ and we can conclude that $\tropx f(x) \in \R^+$
  arguing similarly to the first subcase; assume now that $x >
  \alpha_{+\infty}$, and write  $x = \alpha_{+\infty}(1+ \epsilon)$, with
  $\epsilon > 0$; we define $c_j := b_{j}\alpha_{+\infty}^j$. Observe
  that for all $j$ it holds
  \[
    \tropx f(x) \geq b_jx^{j} = b_{j}(\alpha_{+\infty}(1+\epsilon))^{j}
    = c_j(1+\epsilon)^j.
    \]
 Since the sequence of tropical roots is increasing, its limit
        for the index tending to $+\infty$ is an upper bound. Hence,
        by Lemma~\ref{lem:newtonIsolated} we have that for all $k$ such that $\alpha_k$ exists
	\[
          \left(\frac{b_{j_{k}}}{b_{j_{k+1}}}\right)^{\frac{1}{j_{k+1}-j_{k}}}
          < \alpha_{+\infty} \Leftrightarrow c_{j_k} < c_{j_{k+1}}.
        \]
        Therefore, $c_{j_k}$ is also an increasing sequence, and in
        particular $c_{j_k} > c_{j_\ell}$ for all $k \geq \ell$. As a
        consequence, fixing any $\ell$ such that $\alpha_\ell$ is
        defined,
		\[ \tropx f(x) \geq b_{j_k}x^{j_k} > c_{j_k}
                  (1+\epsilon)^{j_{k}}
                \]
       whose limit when $k \rightarrow + \infty$ is $+ \infty$. Hence, 
       $\tropx f(x) =+\infty$ for all $x > \alpha_{+\infty}$, and 
       $\alpha_{+\infty}$ is on the right boundary of 
       $D$. Since it cannot be a tropical root 
       (otherwise we would have a largest tropical root 
       contradicting the assumption), we have 
       $\tropx f(\alpha_{+\infty}) =+\infty$.\qedhere
\end{itemize}
\end{enumerate}
\end{proof}

\begin{theorem}[Newton polygon for Tropical Laurent series]
\label{thm:newtonPoly}
Let $\tropx f(x)$ be a tropical Laurent series and let
$\newt{\tropx f}$
be its Newton polygon. Then $\tropx f(x)$ has a largest finite
tropical root $\alpha_{p}$ of infinite multiplicity if and only if
$\newt{\tropx f}$ has a rightmost segment of infinite length and
vertex $(j_p, \log b_{j_p})$, with
\[
  \log \alpha_p =  \inf_{i \in I} \sup_{j \geq i} \left( -\frac{\log
      b_j-\log b_i}{j-i} \right) 
\]
Similarly, $\tropx f(x)$ has a smallest finite
tropical root $\alpha_{f}$ of infinite multiplicity if and only if
$\newt{\tropx f}$ has a leftmost segment of infinite length and
vertex $(j_f, b_{j_f})$, with
\[
  \log \alpha_{f} = \sup_{i \in I} \inf_{j \leq i} \left(
    - \frac{\log b_i - \log b_j}{i - j}
  \right). 
  \]
\end{theorem}
\begin{proof}
  Assume $\tropx f(x)$ has a largest tropical root of infinite
  multiplicity $\alpha_p$. By Definition~\ref{def:tropicalroot} this
  corresponds to $D$ being bounded above and closed at its upper
  endpoint.  In addition, recall that the set of indices  $I :=\{j : b_{j} > 0\}$
  is not bounded above thanks to
  Proposition~\ref{prop:nRoots}. Furthermore, the sequence $T$ of the
  indices of the tropical roots is 
  bounded above by Theorem~\ref{thm:asymptotics}: in particular, there
  exists a rightmost point of non-differentiability, say,
  $\alpha_{p-1}$, which is in correspondence to a segment on the
  Newton polygon by Lemma~\ref{lem:newtonIsolated}. We label its rightmost
  segment by $(j_{p}, \log b_{j_{p}})$. Then the rightmost convex hull
  of the points $(j,\log b_j)$ must have a rightmost infinite
  segment\footnote{
    To visualize an example of such rightmost 
    infinite segment, see Figure~\ref{fig:ex-upd-1}
    in Section~\ref{sec:practical-computation}.
  }. It follows from the definition of convex hull that the
  slope of this rightmost segment must be equal to the supremum of all
  the slopes of the segments through $(j, \log b_j)$ and $(j_p,\log b_{j_p})$,
  where the supremum is taken over all the (infinitely many) values of
  $j>j_p$ such that $b_j > 0$. Moreover, for any $j_p < i < j$, it must
  be
  \[
    \frac{\log b_j-\log b_{j_p}}{j-j_p} \leq \frac{\log b_j - \log b_i}{j-i},
  \]
  as otherwise $\alpha_{p-1}$ is not the penultimate tropical
  root. Thus, the opposite of the rightmost slope is
  \[
          \inf_{j \geq j_p} \left( - \frac{\log b_j - \log b_{j_p}}{j-j_p}
          \right) = \inf_{i} \sup_{j \geq i}
          \left(-\frac{\log b_j-\log b_i}{j-i} \right) =: \log \alpha_p.
   \]
The case of a smallest tropical root of infinite
   multiplicity is dealt with analogously.
\end{proof}

As noted in the introduction, many researchers have
used the tropical roots of $\tropx p(x)$ to retrieve bounds on the
localization of the roots of a scalar polynomial $p(\lambda)$ or the eigenvalues of
a matrix polynomial 
$P(\lambda)$~\cite{nosh:15, sh:11}. In the next section, we will generalize
this approach to include (matrix-valued) Laurent series. Albeit technical, the results of the present section allow us to make this step very naturally.

\section{Tropicalization of functions analytic on an annulus }
\label{sec:evsLoc}
The aim of this section is to relate the tropical roots of Laurent
series with the eigenvalues of matrix-valued functions analytic
on an annulus. More 
precisely, given the Laurent series
\[
  F(\lambda) := \sum_{j\in\Z}B_{j}\lambda^{j},
  \qquad 
  B_j \in \mathbb{C}^{m \times m},
  \]
we consider $F(\lambda)\colon \Omega \to \Cnn$, where $\Omega$
is the largest annulus where the sum defining $F(\lambda)$ is convergent. 
We denote this set as $\Omega:= \annop(\rr, \RR)$, and the radii 
$\rr$ and $\RR$ are uniquely determined by the decay rate of 
$\lVert B_j \rVert$ for $j \to \pm \infty$. 

\begin{lemma} \label{lem:laurentradius}
  Let $F(\lambda) = \sum_{j\in\Z}B_{j}\lambda^{j}$ be a Laurent series 
  with $B_j \in \mathbb{C}^{m \times m}$, and $\norm{\cdot}$ any 
  matrix norm. Then, the series is convergent in the annulus 
  $\Omega:= \annop(\rr, \RR)$, where 
  \begin{equation}
  \label{eq:laurentRadii}
    \RR^{-1} = \limsup_{j \to \infty} \norm{B_{j}}^{1/j}, \qquad 
    \rr =\limsup_{j \to \infty} \norm{B_{-j}}^{1/j},
  \end{equation}
  where we employ the usual convention $\infty^{-1} = 0$.
\end{lemma}

\begin{proof}
  We argue that the radii of convergence for
  $F(\lambda)$ are the same as those of the scalar series
  $f(\lambda):=\sum_{i \in \Z} b_i \lambda^i$, where $b_i = \norm{B_i}$,
  and that this holds independently of the choice of the norm
  $\norm{\cdot}$. This follows by observing that for all $\lambda$
  \[ \norm{F(\lambda)} \leq \abs{f(\lambda)} \leq c_n \cdot
    \sum_{i,j=1}^n \abs{F(\lambda)_{ij}} \leq c_n \cdot C_n \cdot
    \norm{F(\lambda)}, \] where $c_n,C_n$ are some positive constants
  (both allowed to depend on $n$) and we have used the fact that all
  matrix norms are equivalent.
\end{proof}

Note that the
definition covers Taylor series, for which $\rr = 0$ and 
$\RR$ is the radius of convergence, and 
polynomials, for which $\rr = 0$ and $\RR = \infty$. 

For matrix polynomials, where $\Omega = \C$, the number
of eigenvalues is always determined by the size $n$ of the matrix
coefficients $B_{j}$ and the degree $d$; in
contrast, a Laurent series 
$F(\lambda)$ may have zero, finite or
countably infinite eigenvalues inside $\Omega$. 
We also assume that $F(\lambda)$ is regular, i.e., $\det F(\lambda) \not \equiv 0$. This implies that the eigenvalues of $F(\lambda)$ coincide with the roots of the scalar function $\det F(\lambda)$.
Given an arbitrary, but fixed, subordinate matrix norm $\norm{\cdot}$, to the matrix Laurent series $F(\lambda)$ we
associate the max-times Laurent series
\begin{equation}
\label{eq:Ftropic}
\tropx F(x) := \sup_{j\in \Z}(\norm{B_{j}}x^{j}) = \sup_{j\in \Z}(b_{j}x^{j}).
\end{equation}
While a priori the tropical roots of $\tropx F(x)$ depend on the choice of the norm, we will prove in Theorem \ref{thm:radiiandtrop} that the quantities $\alpha_{\pm \infty}$ do not, and are related to the radii of convergence of $F(\lambda)$. It is clear that a max-plus tropicalization can be defined analogously, and we omit the details.

\subsection{Tropical roots and radii of convergence}

The extension of the definition of tropical roots from tropical polynomials to tropical Laurent series has been
discussed in Section~\ref{sec:tropicalseries}. We noted that most roots
can still be defined as the points of non-differentiability of the
piecewise polynomial function $\tropx F(x)$, and can be computed
through the Newton polygon (see Theorem~\ref{thm:newtonPoly}). This
may pose a computational challenge, as we need to compute the convex
hull of a countable set of points; we discuss how to perform this task
in detail in Theorem~\ref{thm:truncation}. 
We denote the ratio between two consecutive roots with
\begin{equation}
  \label{eq:delta}
  \delta_{j} = \frac{\alpha_{j}}{\alpha_{j+1}} < 1, \qquad \text{for }
  j \in \Z.
\end{equation}
Intuitively, one can
visualize $\delta_{j}$ on the associated Newton polygon. The smaller
$\delta_{j}$ is, the spikier the polygon is in $(j_{k},
\log{B_{j_{k}}})$. We illustrate this below with two scalar examples.

\begin{example}
  \label{ex:toy1}
  Consider the function
  \[
    f(\lambda) =
  \frac{15}{(1-3\lambda)(\lambda-2)},
\]
which is holomorphic in $\annop(1/3,2)$. It is not difficult to see that
\[
  \begin{split}
  f(\lambda) &= 6(\cdots + \frac{1}{3\lambda^{2}}+\frac{1}{\lambda} +
  \frac{1}{2} + \frac{\lambda}{4} + \cdots)\\
  &= \sum_{j\in \Z}b_{j}\lambda^{j}
\end{split}
\]
where $b_{j} = 2\cdot3^{j+2}$ if $j < 0$, and $b_{j} = 3\cdot2^{-j}$
for $j \geq 0$. The (truncated) Newton polygon is given 
in Figure~\ref{fig:toy1} and we can easily detect that there are only two
roots, $\alpha_{-\infty} = 1/3$ and $\alpha_{\infty} = 2$, using 
the characterization from Lemma~\ref{lem:newtonIsolated}. In addition,
one can check that $\tropx f(\alpha_{+\infty}) = 3$ and $\tropx
f(\alpha_{-\infty}) = 18$, hence $\tropx f(x)$ is a well-defined real-valued
function in the domain $D =
\interval{\alpha_{-\infty}}{\alpha_{\infty}}$, in concordance with
both Definition~\ref{def:tropicalroot} and
Theorem~\ref{thm:asymptotics}. 

\begin{figure}
    \centering\includegraphics[width=0.95\textwidth]{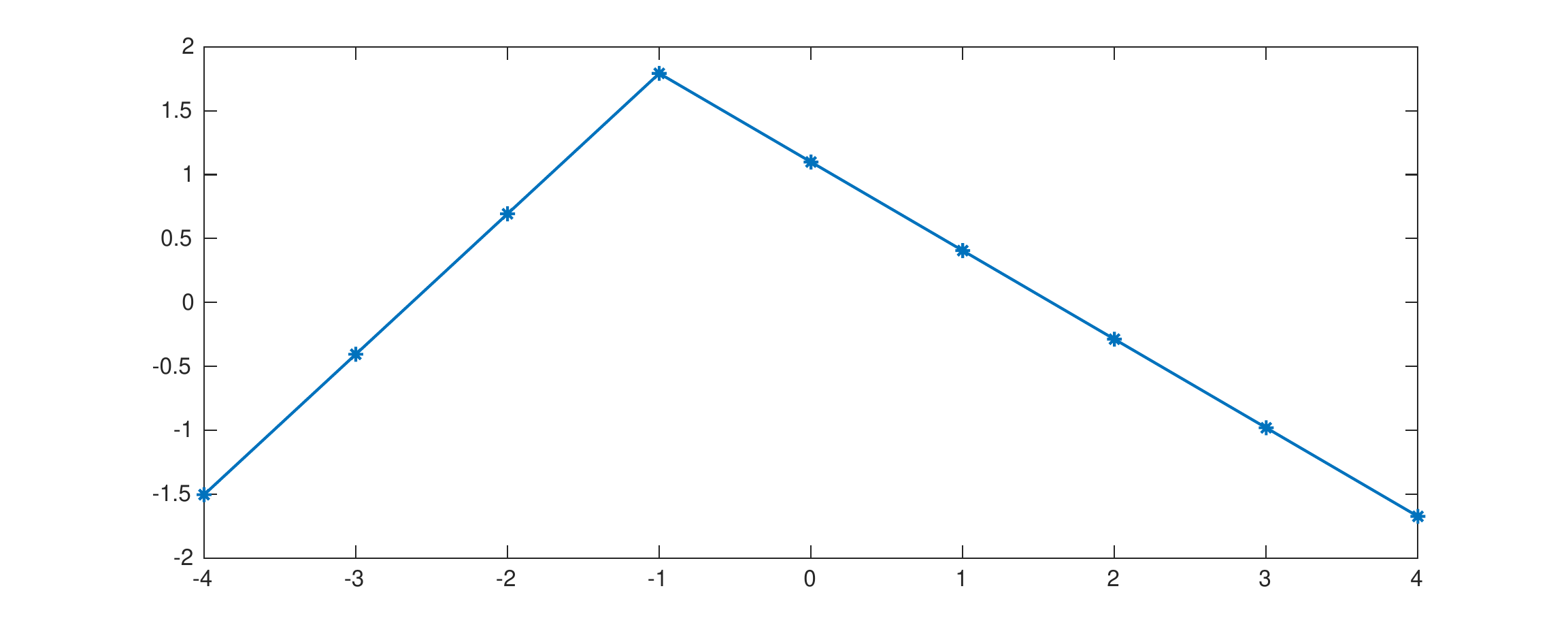}
    \caption{The (truncated) Newton polygon of $\tropx f(x)$. The two tropical
    roots $\alpha_{\pm\infty}$ are in correspondence with the two slopes.}
    \label{fig:toy1}
\end{figure}

\end{example}

\begin{example}
  We consider again the function
  \[
  f(\lambda) = \sum_{j=1}^\infty \eu^{H_j} \lambda^j \Rightarrow \tropx f(x) =
  \sup_{j \geq 0} (\eu^{H_j}x^j).
  \]
  defined in Example~\ref{ex:2}. We have that the domain of $f(\lambda)$ is
  the open disk
  $\tset = \disk(1)$, while the  domain $D$ of
  $\tropx f(x)$ is $D=\interval[open right]{0}{1}$. In
  Figure~\ref{fig:harmonNewt} we plotted its truncated Newton
  polygon. There we can see that the slopes converge from below to
  $0$, in concordance with the fact that the nonzero roots are
  $\alpha_{j} = \eu^{-1/j}\to 1$, as $j$ goes to infinity. 
\begin{figure}
    \centering\includegraphics[width=0.95\textwidth]{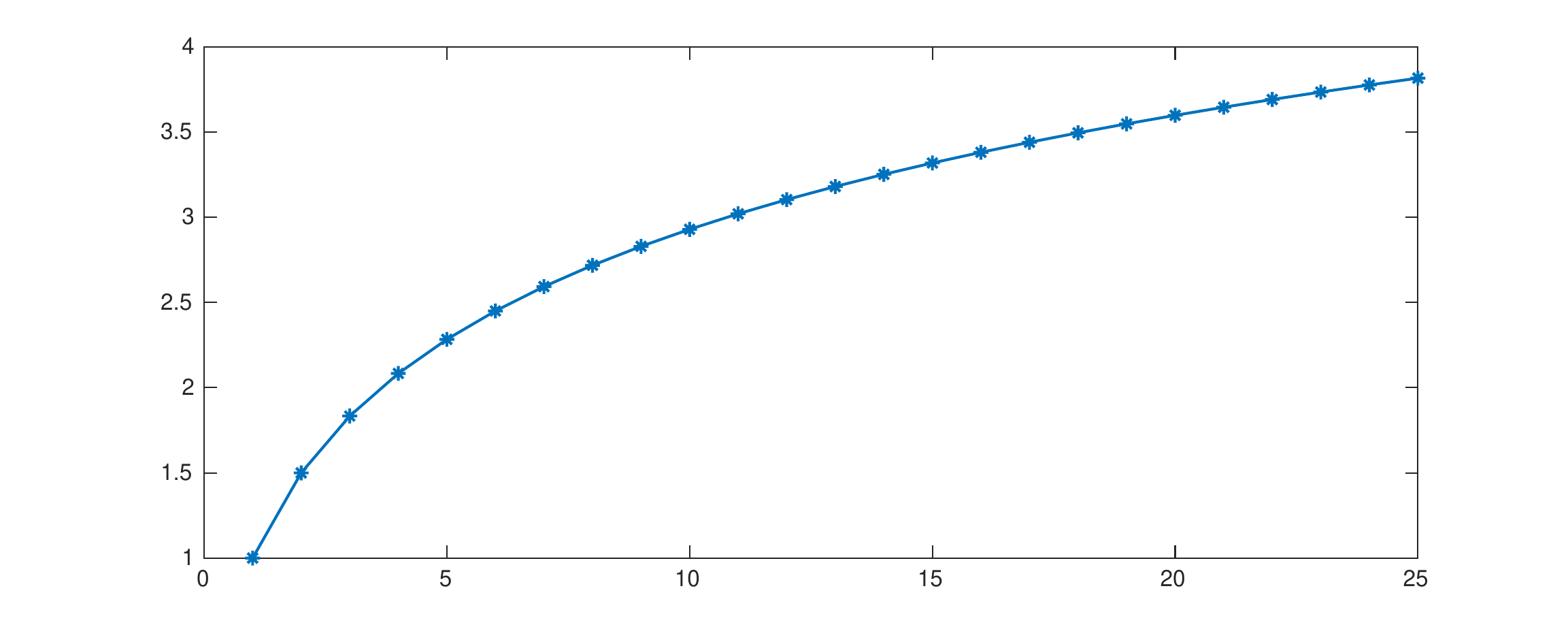}
    \caption{The (truncated) Newton polygon of $\tropx f(x)$. The
      slopes converge from below to $0$. Note that in this case 
      the term polygon is an abuse of notation, since the convex 
      set has an infinite number of edges.}
    \label{fig:harmonNewt}
\end{figure}

\end{example}

In both the examples above, the tropical roots not only converge
to (or are) the endpoints of the domain $D$ of $\tropx f(x)$, but the latter coincide with the radii of convergence of $f(\lambda)$. Once again, this does not
happen by chance, as explained in the following theorem.

\begin{theorem}\label{thm:radiiandtrop}
  Consider the $n \times n$ Laurent matrix-valued function
  $F(\lambda) = \sum_{i\in \Z}B_{i}\lambda^{i}$, holomorphic in the
  open annulus $\Omega=\annop(\rr, \RR)$, where $\rr$ and $\RR$ are
  defined in~\eqref{eq:laurentRadii}. Let $(\alpha_{k})_{k\in T}$ be the
  sequence of distinct tropical roots  of the tropicalization $\tropx F(x)$
and let $\alpha_{\pm\infty}$ be
  the quantities of Definition~\ref{def:ainfty}. Then $\RR =
  \alpha_{+\infty}$ and $\rr = \alpha_{-\infty}$.
\end{theorem}

\begin{proof}
  Thanks to Lemma~\ref{lem:laurentradius}, the series is convergent 
  in $\annop(R_1, R_2)$, defined as in \eqref{eq:laurentRadii}. 
  Then, we may equivalently consider $f(\lambda)$ and the proof becomes similar to the one of Theorem~\ref{thm:asymptotics}. As
  usual, we prove this statement only for $\RR$ as the argument for $\rr$ is identical.
\begin{enumerate}
\item\label{item:1} \emph{There is a largest tropical root and it has finite
    multiplicity.} We have $\alpha_{+\infty} = +\infty$,
 and from Proposition~\ref{prop:nRoots} it follows $b_{j} = 0$ for $j
 > 0$ and large enough. Hence $\RR = +\infty$.

 \item\emph{There is a largest tropical root with infinite
     multiplicity}, which we denote by 
   $\alpha_{p} = S_{2}$. By definition
   of tropical roots it holds
  \[
    b_{j}\alpha_{p}^{j} \leq b_{j_p}\alpha_{p}^{j_p}
  \]
  therefore
  \[
   \frac{b_{j}}{b_{j_{p}}} \leq S_{2}^{(j_p-j)}.
  \]
  This implies that $F(\lambda)$ is well-defined for every
  $\abs{\lambda} \leq S_{2}$ by~\eqref{eq:laurentRadii}, therefore
  $S_{2} \leq \RR$. We now claim that $S_{2} \geq R_{2}$. Let
  $(b_{j_{k}})_{k\in \N}$ be the subsequence of $(b_{j})_{j\in \N}$
  obtained by only keeping the indices corresponding to a curve that
  attains the supremum in the definition of $\tropx F(x)$. Then
  \[
    \frac{1}{R_{2}} \geq \limsup_{j\to \infty} \abs{b_{j}}^{1/j} \geq
    \limsup_{k\to \infty}b_{j_{k}}^{1/j_{k}}.
  \]
  Let $j_{k'} = p$ be the index corresponding to the last tropical
  root with infinite multiplicity $\alpha_{p}$. Then for all $k > k'$
  we have
  \[
    x^{j_{k}}b_{j_{k}} \geq x^{p}b_{p},
  \]
  for every $x \geq \alpha_{p}$ by definition of tropical root. Hence,
  \[
  b_{j_{k}} \geq b_{p}S_{2}^{p-j_{k}}\Rightarrow
  b_{j_{k}}^{1/j_{k}} \geq \frac{(S_{2}^{p}b_{p})^{1/j_{k}}}{S_{2}}.
\]
It follows that
\[\frac{1}{\RR} \geq \limsup_{k\to
    \infty}\frac{(S_{2}^{p}b_{p})^{1/j_{k}}}{S_{2}} =
  \frac{1}{S_{2}}.
\]

\item \emph{There is no largest tropical root.} 
If there is no largest tropical root, then $\alpha_{+\infty} =
\lim_{j\to+\infty}\alpha_{j}$. If $\alpha_{+\infty} \in\R^{+}$,
then one can argue as in case (2) above and show that $\RR =
\alpha_{+\infty}$. If $\alpha_{+\infty} = +\infty$, then the sequence
of tropical roots is unbounded. But if it is unbounded, then again by the
  proof of case (2) given any tropical root $\alpha_k$ the
  power series defining $F(\lambda)$ must converge for all
  $\abs{\lambda} < \eu^{\alpha_k}$, hence $\RR=+\infty$.\qedhere
\end{enumerate}
\end{proof}

\subsection{Eigenvalue localization for Laurent series}
In this section we generalize to matrix-valued Laurent series the localization theorems derived in \cite{nosh:15} for matrix-valued polynomials. We are going to follow the steps of~\cite{nosh:15};
thus, we simply state those results whose proof literally does not change, and we prove the ones that need some adjustments due to the infinite
number of indices.  As expected, the core idea is Rouché's Theorem,
which we recall for clarity.

\begin{theorem}[Rouché's Theorem]
\label{thm:rouche}
  Let $S, Q : \tset_0 \to \Cnn$ be meromorphic matrix-valued functions,
  where $\tset_0$ is an open connected subset of $\C$. Assume that
  $S(x)$ is nonsingular for every $x$ on the simple closed curve
  $\Gamma \subset \tset_{0}$. Let $\norm{\cdot}$ be any operator norm on
  $\Cnn$ induced by a vector norm on $\C^{n}$. If
  $\norm{S(x)^{-1}Q(x)} < 1$ for every $x\in \Gamma$, then
  $\det (S+Q)$ and $\det (S)$ have the same number of zeros minus
  poles inside $\Gamma$, counting multiplicities.
\end{theorem}

When we will apply Rouché's Theorem to meromorphic functions
$F(\lambda)$, hence their Laurent expansions will only have a finite
amount of negative indices.  Under these settings, $\tset_{0}$ will be
the disk $\disk(\RR)$, and, as in~\cite{nosh:15}, we are using a given
tropical root $\alpha_{j}$ as a parameter scaling, with
$\lambda = \alpha_{j}\mu$ and $\Ft(\mu)$ equal to
\begin{equation}
\label{eq.scalfun}
        \left(\tropx F(\alpha_j)\right)^{-1} F(\lambda) =
        \left(\norm{B_{k_{j-1}}}\alpha_j^{{k_{j-1}}}\right)^{-1}
        F(\alpha_j\mu) = \sum_{i=0}^d\Bt_i\mu^i=:\Ft(\mu),
\end{equation}
where
\begin{equation}\label{def.Bi}
\Bt_i=\left(\norm{B_{k_{j-1}}}\alpha_j^{{k_{j-1}}}\right)^{-1}B_i\alpha_j^i.
\end{equation}
\begin{lemma}[{\cite[Lemma 2.2]{nosh:15}}]
\label{prop:norms}
The norms of the coefficients $\Bt_{i}$~\eqref{def.Bi} of the scaled function have the
following properties:
\[
\norm{\Bt_i}\leq
  \begin{cases}
  {\delta_{j-1}}^{{k_{j-1}}-i} & \text{if} \quad   i<{k_{j-1}},\\
   1 &\text{if} \quad k_{j-1}< i < k_j, \\
  {\delta_{j}}^{i-{k_j}} & \text{if} \quad i > k_j,
  \end{cases}
\qquad
\norm{\Bt_{k_{j-1}}}=\norm{\Bt_{k_j}}=1.
\]
\end{lemma}

With the goal of invoking Theorem~\ref{thm:rouche}, we decompose
$\Ft(\mu)$ as the sum of
\begin{equation}
\label{eq:decompo}
S(\mu) = \sum_{i=k_{j-1}}^{k_{j}}\Bt_{i}\mu^{i}, \qquad
Q(\mu) = \sum_{i\not\in\{k_{j-1},\dots, k_{j}\}}\Bt_{i}\mu^{i}.
\end{equation}
We use the notation $\kappa(A) := \norm{A} \norm{A^{-1}}$ 
to denote the condition number of any nonsingular matrix $A$. 
Thanks to the following results, we can localize the 
eigenvalues of $S(\lambda)$.
\begin{lemma}[{\cite[Lemma~2.3]{nosh:15}}]
  \label{prop:eigall}
  Let $P(\lambda) = \sum_{j=0}^{\ell} B_j \lambda^j$ with
  $B_0,B_{\ell} \neq 0$ be a regular matrix polynomial.  Then every
  eigenvalue of $P(\lambda)$ satisfies
\[
  \left(1+\cond(B_0)\right)^{-1}\alpha_1 \leq \abs{\lambda}\leq \left(1+\cond(B_\ell)\right)\alpha_q, 
\]
where $\alpha_1, \alpha_q$ are the smallest and largest finite tropical 
roots, respectively. Furthermore, if  $B_0$ and $B_\ell$ 
are invertible, both inequalities are strict.
\end{lemma}



\begin{lemma}[{\cite[Lemma 2.4]{nosh:15}}]
  \label{prop:eig}
  Let $m_{j}:= k_{j}-k_{j-1}$. Then, if $B_{k_{j-1}}$ and $B_{k_{j}}$ are nonsingular, then the
  $n(k_{j}-k_{j-1})$ nonzero eigenvalues of $S(\lambda)$ are located
  in the open annulus
  \mbox{$\annop((1+\cond(B_{k_{j-1}}))^{-1}, 1+\cond(B_{k_{j}}))$}.
\end{lemma}
\begin{proof}
  If $k_{j-1} \geq 0$, then the lemma is in fact equivalent  to~\cite[Lemma 2.4]{nosh:15}. Otherwise, consider the polynomial
  $\wt{S}(\lambda) = \lambda^{-k_{j-1}}S(\lambda)$, which we can apply
  the mentioned above to. This implies that $\wt{S}(\lambda)$ has
  $nm_{j}$ nonzero eigenvalues in
  $\annop((1+\cond(B_{k_{j-1}}))^{-1}, 1+\cond(B_{k_{j}}))$, which
  yields that $S(\lambda)$ has $nm_{j}$ nonzero eigenvalues and
  $nk_{j-1}$ poles in zero.
\end{proof}
\begin{lemma}[{\cite[Lemma 2.5]{nosh:15}}]
\label{lem:genb}
The following inequalities hold for $Q(\mu)$ and the inverse of $S(\mu)$ in
\eqref{eq:decompo},
\begin{eqnarray*}
\norm{S(\mu)^{-1}}&\le& \left\{
\begin{array}{ll}
\displaystyle{\frac{\cond(B_{k_j-1})
\abs{\mu}^{-k_{j-1}}(1-\abs{\mu})}{1-\abs{\mu}\big(1+\cond(B_{k_j-1})(1-\abs{\mu}^{m_j})\big)}}
&
\mbox{if $0<\abs{\mu}\le \big(1+\cond(B_{k_j-1})\big)^{-1}$,}\\
\mystrut{1cm}
\displaystyle{\frac{\cond(B_{k_j}) \abs{\mu}^{-{k_j}}(\abs{\mu}-1)}{\abs{\mu}-1-\cond(B_{k_j})(1-\abs{\mu}^{-{m_j}})}}&
\mbox{if $\abs{\mu}\ge 1+\cond(B_{k_j})$}.\\
\end{array}
\right.,
\\
\mystrut{1.5cm}
\norm{Q(\mu)}&\le&
\frac{{\delta_{j-1}} \abs{\mu}^{k_{j-1}}}{\abs{\mu}-{\delta_{j-1}}}
+\frac{{\delta_{j}} \abs{\mu}^{k_j+1}}{1-{\delta_{j}} \abs{\mu}}
\qquad \mbox{if $\delta_{j-1}<\abs{\mu}<\displaystyle{\frac{1}{\delta_j}}$}.
\end{eqnarray*}
\end{lemma}
\begin{proof}
  Given that $S(\mu)$ is a Laurent polynomial, the proof about
  $\norm{S(\mu)}^{-1}$ is identical to~{\cite[Lemma
    2.5]{nosh:15}}. The steps for $\norm{Q(\mu)}$ are not
dissimilar either, but the details are slightly different due to
  $Q(\mu)$ being a Laurent series.  Assume that
  $\delta_{j-1}<\abs{\mu}<\delta_j^{-1}$.  Using \eqref{eq:decompo}
  and Lemma~\ref{prop:norms} we have that
  \begin{align*}
    \norm{Q(\mu)}
     &\le \sum_{i<{k_{j-1}}-1}
       {\delta_{j-1}}^{{k_{j-1}}-i}\abs{\mu}^i+ \sum_{i>{k_{j}}+1}
       {\delta_{j}}^{i-{k_{j}}}\abs{\mu}^i\\
    &\leq\frac{{\delta_{j-1}}(\abs{\mu}^{k_{j-1}}-{\delta_{j-1}}^{k_{j-1}})}{\abs{\mu}-{\delta_{j-1}}}
    + \frac{{\delta_{j}} \abs{\mu}^{{k_{j}}+1}(1-({\delta_{j}}
                                               \abs{\mu})^{-{k_{j}}})}{1-{\delta_{j}}
                                               \abs{\mu}} 
    \end{align*}
    and the bound in the lemma follows since
$\delta_{j-1}<\abs{\mu}<\displaystyle{\frac{1}{\delta_j}}$.
\end{proof}

After the upcoming  technical lemma, we will have all the tools to
generalize~\cite[Theorem~2.7]{nosh:15} to Laurent matrix-valued
functions. Not every statement translates exactly under
these new settings, given the possible presence of the poles, but the
ability of locating the eigenvalues of $F(\lambda)$ thanks to the
associated tropical roots still holds true.

\begin{lemma}[{\cite[Lemma 2.6]{nosh:15}}]
\label{lem.key}
Given $c,\delta>0$ such that $\delta\le (1+2c)^{-2}$, the
quadratic polynomial
\[
        p(r) = r^2 -\bigg(2+\frac{1-\delta}{\delta(1+c)}\bigg)r+\frac{1}{\delta}
\]
has two real roots
\[
        f := f(\delta,c)=  \frac{(1+2c)\delta+1-\sqrt{(1-\delta)(1-(1+2c)^2\delta)}}
             {2\delta(1+c)},
        \qquad
        g = (\delta f)^{-1},
\]
with the properties that
\begin{enumerate}
\item[\rm(i)] $1<1+c\le f\le g$,
\item[\rm(ii)]
$\displaystyle{\frac{1}{f-1} + \frac{1}{g-1} = \frac{1}{c}}$.
\end{enumerate}
\end{lemma}

\begin{theorem}
  \label{thm:evsLoc}
  Let $\ell^{-} >-\infty$ and let
  $F(\lambda) = \sum_{j=\ell^{-}}^{\infty}B_{j}\lambda^{j}$ be a
  regular, meromorphic Laurent function, analytic in the open annulus
  $\Omega := \annop(\rr, \RR)$. For every $j \in \Z$, let
  $f_j = f(\delta_j,\cond(B_{k_j}))$, where $f(\delta,c)$ is defined
  as in Lemma \ref{lem.key}, and $g_j = (\delta_j f_j)^{-1}$. Then:
\begin{enumerate}
\item If $\delta_{j} \leq (1 +2 \cond(B_{k_{j}}))^{-2}$,
  then $F(\lambda)$ has exactly $nk_{j}$ eigenvalues minus poles inside the
  disk $\disk((1+2\cond(B_{k_j}))\alpha_{j})$ and it does not have any eigenvalue
  inside the open annulus $\annop((1+2\cond(B_{k_j}))\alpha_{j},
  (1+2\cond(B_{k_j}))^{-1}\alpha_{j+1})$. 
\item For any $j < s$,  if $\delta_{j} \leq (1 +2 \cond(B_{k_{j}}))^{-2}$ and
  $\delta_{s} \leq (1 +2 \cond(B_{k_{s}}))^{-2}$, then
  $F(\lambda)$ has exactly $n(k_{s}-k_{j})$ eigenvalues  inside
  the closed annulus $\ann((1+2\cond(B_{k_j}))^{-1}\alpha_{j+1}, (1+2\cond(B_{k_s}))\alpha_{s})$.
\end{enumerate}
\end{theorem}
\begin{proof}
  As a preliminary step, note that for a fixed $c \geq 1$ and
  $\delta \leq (1+2c)^{-2}$, the function $f(\delta,c)$ of
  Lemma~\ref{lem.key} is increasing and attains its maximum, which is
  $1+2c$, at $\delta = (1+2c)^{-2}$.  Therefore
it holds  $f(\delta_{j}, \cond(B_{k_{j}})) \leq 1+2\cond(B_{k_j})$ for
  $\delta_{j}\leq (1+2\cond(B_{k_{j}}))^{-2}$.
\begin{enumerate}
\item  Assume  that $\delta_j\le(1+2\cond(B_{k_j}))^{-2}$ and
 partition $\Ft(\mu)$ as in \eqref{eq:decompo}.
Consider now $r$ such that
\begin{equation}
\label{def.r}
        1+\cond(B_{k_j})<r<1/\delta_j.
\end{equation}
Note that $r$ is well defined because
$\delta_j\leq(1+2\cond(B_{k_j}))^{-2}<(1+\cond(B_{k_j}))^{-1}$.
Thanks to Lemma~\ref{prop:eig}, we have that $S(\mu)$ is nonsingular
on the circle $\Gamma_r=\{\mu\in\C: \abs{\mu}=r\}$.  As soon as we can
apply Rouché's Theorem~\ref{thm:rouche} to $\Ft(\mu) = S(\mu)+Q(\mu)$
and $\Gamma_{r}$, we will have concluded. Therefore we have to check that
$\norm{S(\mu)^{-1}Q(\mu)}<1$ for all $\mu\in\Gamma_r$. Since
\begin{equation}\label{ineq.r0}
\delta_{j-1}<1<1+\kappa(B_{k_j})<r<1/\delta_j,
\end{equation}
we can simply apply the bounds retrieved in
Lemma~\ref{lem:genb}. It follows
\begin{align*}
   \norm{S(\mu)^{-1}Q(\mu)}
   &\le \norm{S(\mu)^{-1}}\norm{Q(\mu)}\\
   &\le \frac{r^{-k_j}(r-1)\cj}
       {r-1-\cj(1-r^{-{m_j}})}\left(\frac{{\delta_{j-1}} r^{k_{j-1}}}
   {r-{\delta_{j-1}}}+
   \frac{{\delta_{j}} r^{k_j+1}}{1-{\delta_{j}} r}\right).
\end{align*}
The latter bound is less than $1$ if
\begin{equation*}
   \frac{{\delta_{j-1}} r^{-{m_j}}}
   {r-{\delta_{j-1}}}+\frac{{\delta_{j}} r}{1-{\delta_{j}} r}
   <
   \frac{r-1-\cj(1-r^{-{m_j}})}{(r-1)\cj},
\end{equation*}
or equivalently, if
\begin{equation*}\label{eq:polyineq}
  \frac{{\delta_{j}} r}{1-{\delta_{j}} r}<
  \frac{r-1-\cj}{(r-1)\cj}+r^{-{m_j}}
  \left(\frac{1}{r-1} -\frac{\delta_{j-1}}{r-{\delta_{j-1}}}\right).
\end{equation*}
Since
$\frac{1}{r-1}>\frac{{\delta_{j-1}} }{r-{\delta_{j-1}}}$, the last inequality holds when
$\frac{{\delta_{j}} r}{1-{\delta_{j}} r}<\frac{r-1-\cj}{(r-1)\cj}$,
which is equivalent to  $p(r)<0$, where $p(x)$ is the polynomial
introduced in Lemma~\ref{lem.key} with
$\delta=\delta_j$ and $c=\cj$. Lemma~\ref{lem.key} assures us that
$p(r)$ is negative for the values of $r$ such that
\begin{equation}
  \label{ineq.r}
        f_j<r<g_j,
\end{equation}
given that $f_j$ and $g_j$ are the two roots of $p(x)$.  The same
lemma also tells us that $f_j\geq 1+\cj$ and $g_j\leq
(\delta_j)^{-1}$, hence~\eqref{ineq.r} is sharper than~\eqref{ineq.r0}.  
In particular, note that for any $\abs{\mu} = r$ 
where $r = f_j$ or $r = g_j$ the 
upper bound for $\norm{S(\mu)^{-1} Q(\mu)}$ is equal to $1$. 
Therefore, such $\mu$ belongs to the domain of analiticity, and 
we have $R_1 \leq f_j < g_j \leq R_2$. 
Finally, by Rouché's Theorem this implies that $S(\mu)$ and $\Ft(\mu)$
have the same number of eigenvalues minus poles, i.e., $nk_{j}$,
inside the disk $\disk(r)$ for any $r$ such that $f_{j} < r <
g_{j}$. Furthermore, there are no eigenvalues in the open annulus
$\annop(f_{j}, g_{j})$, because the quantity of eigenvalues minus
poles is constant and there cannot be poles in $\annop(f_{j},
g_{j})$. The thesis then follows from the scaling
$\lambda = \mu\alpha_{j}$ and from the preliminary point.
\item By the proof of the previous item, if $\delta_{j} \leq (1 +2 \cond(B_{k_{j}}))^{-2}$ then
  $F(\lambda)$ has $nk_{j}$ eigenvalues minus poles inside
  $D((1+2\cond(B_{k_{j}}))\alpha_{j})$ and no eigenvalues in
  $\annop((1+2\cond(B_{k_j}))\alpha_{j},
  (1+2\cond(B_{k_j}))^{-1}\alpha_{j+1})$. A similar statement holds
  if $\delta_{s} \leq (1 +2 \cond(B_{k_{s}}))^{-2}$. Given that poles
  cannot lie there, this implies that $F(\lambda)$ has exactly
  $n(k_{s} - k_{j})$ eigenvalues inside the closed  annulus
  $\ann((1+2\cond(B_{k_j}))^{-1}\alpha_{j+1},
  (1+2\cond(B_{k_s}))\alpha_{s})$. \qedhere
\end{enumerate}
\end{proof}

When the non-zero coefficients with negative (resp. positive)
indices are a finite number, we may also find an exclusion (resp. inclusion) 
disc centered at zero. This may be seen as a generalization of 
Lemma~\ref{prop:eigall}. 

\begin{theorem} \label{thm:lower-upper}
  Let $F(\lambda) = \sum_{j=\ell^-}^{\ell^+} B_j \lambda^j$ 
  be a (classic) Laurent series. 
  When $\ell^- > -\infty$ and 
  $B_{\ell^-}$ is non singular, then $F(\lambda)$ has $\ell^- n$ eigenvalues 
  minus poles at zero, and the other eigenvalues of $F(\lambda)$ satisfy
  \[
    \left(1+\cond(B_{\ell^-})\right)^{-1}\alpha_1 \leq \abs{\lambda}
  \]
  Similarly, if $\ell^+ < \infty$ and $B_{\ell^+}$ is non singular, 
  then  every eigenvalue satisfies
  \[
    \abs{\lambda}\leq \left(1+\cond(B_{\ell^-})\right)\alpha_q, 
  \]
  where $\alpha_q$ is the maximum tropical root. 
\end{theorem}

\begin{proof}
  The proof follows the ideas in \cite[Lemma 4.1]{higham2003bounds}. 
  We consider the 
  case $\ell^+ < \infty$ first. The definition of $\alpha_q$ implies 
  that:
  \[
    \norm{B_i} \leq \alpha_q^{\ell^+-i} \norm{B_{\ell^+}}, \qquad i \leq \ell^+. 
  \]
  Assume for a contradiction that there is an eigenvalue 
  $|\lambda| > \left(1+\cond(B_{\ell^-})\right)\alpha_q$. Then, we 
  may consider a normalized eigenvector $\norm{x} = 1$, and write 
  \begin{align*}
    \norm{F(\lambda) x} &\geq \norm{B_{\ell^+} |\lambda|^{\ell^+} x}  -
      \sum_{i < \ell^+} \norm{B_i} |\lambda|^i  \\
      &\geq |\lambda|^{\ell^+} \left( 
        \norm{B_{\ell^+}^{-1}}^{-1} 
        - \sum_{i < k} \norm{B_i} |\lambda|^{i - \ell^+} 
      \right) \\
      &\geq |\lambda|^{\ell^+} \left( 
        \norm{B_{\ell^+}^{-1}}^{-1} 
        - \sum_{i < k} \norm{B_{\ell^+}} \alpha_q^{\ell^+ - i}
          |\lambda|^{i - \ell^+} 
      \right) \\
      &= |\lambda|^{\ell^+} \norm{B_{\ell^+}^{-1}}^{-1} 
        \left(
          1 - \kappa(B_{\ell^+}) \sum_{i < k} 
            \left[ \frac{\alpha_q}{|\lambda|} \right]^{\ell^+ - i}
        \right) \\
      &= |\lambda|^{\ell^+} \norm{B_{\ell^+}^{-1}}^{-1} 
      \left(
        1 - \kappa(B_{\ell^+}) 
        \frac{\alpha_q}{|\lambda| - \alpha_q}
      \right) > 0
  \end{align*}
  where the last inequality follows assuming 
  that $|\lambda| > \left(1+\cond(B_{\ell^-})\right)\alpha_q$, 
  as we did. Hence, we have the sought 
  upper bound for $|\lambda|$. 

  If we have $\ell^- > -\infty$ and $B_{\ell^-}$ nonsingular, we can 
  rewrite $F(\lambda)$ as 
  \[
    F(\lambda) = \lambda^{\ell^-} P(\lambda), 
  \]
  where $P(\lambda)$ is a Taylor series (or matrix polynomial) with 
  $\det P(0) \neq 0$. Hence, we conclude that $F(\lambda)$ has an
  eigenvalue (or pole, depending on the sign of $\ell^-$) of the desired algebraic 
  multiplicity. Applying the above reasoning to $F(\frac 1\lambda)$ yields 
  the lower bound for the remaining eigenvalues. 
\end{proof}

\subsection{Practical computation of the Newton polygon}
\label{sec:practical-computation}
Given a tropical Laurent series, computing its Newton polygon may 
be challenging, as the customary Graham Scan algorithm \cite{gr:72}
could require an infinite number of comparisons. 

We now present two strategies for addressing this problem.  First, we
may assume that the norms of the coefficients $B_j$ can be computed
easily and cheaply (which is indeed true, for example, for the $1$ and
the $\infty$ operator norms) and we show that the finite truncation of
the Newton polygon ``converge'' (in an appropriate sense) to the
infinite one. Second, we consider the case where the Laurent series is
the modification of a known function, of which only a finite number of
coefficients are altered. This is not unrealistic: for example, in the
context of delay differential equations it is typical to deal with
functions that are polynomials in both $\lambda$ and
$e^\lambda$~\cite{ja:11, jami:10}.

\begin{theorem}[Truncation of the Newton polygon]
\label{thm:truncation}
Let $F \colon \Omega \to \Cnn$ be a Laurent function on $\Omega =
\annop{(\rr,\RR)}$ and let $\tropx F(x)$ be its tropicalization. Let
$\tropx F_{d}(x) = \bigoplus_{0\leq k \leq d}\norm{B_{k}}x^{k}$ be any
``right-finite'' truncation and let $C_2 >0$ be a constant such that
$\norm{B_{k}} < C_2\RR^{-k}$ for any $k > 0$. Then for any two
consecutive indices $i, j \in \newti{\tropx F_{d}}$ such that $\norm{B_{j}}
\geq \norm{B_{i}}\RR^{-(j-i)}$, let $\ell_{2}$ be defined as
\[
  \ell_{2} := {\frac{(j-i)\log C +i\log\norm{B_{j}}
      -j\log\norm{B_{i}}}{(j-i)\log\RR + \log\norm{B_{j}} -
      \log\norm{B_{i}}}}
\]
and suppose moreover that for all $j+1 \leq k \leq \ell_{2}$ it holds
\begin{equation}
\label{eq:thm:truncation}
  \norm{B_{k}} \leq \norm{B_{j}}^{\frac{k-i}{j-i}}\norm{B_{i}}^{-\frac{k-j}{j-i}}.
\end{equation}
Then $i$, $j$ are also consecutive indices in $\newti{\tropx
  F}$.

Similarly, let $\tropx F_{-d}(x) = \bigoplus_{0\leq k \leq
  d}\norm{B_{-k}}x^{-k}$ be a ``left-finite'' truncation and let $C_1 >0$ be a constant such that
$\norm{B_{-k}} < C_1\rr^{k}$ for any $k > 0$. Then for any two
consecutive indices $ -i, -j \in \newti{\tropx F_{-d}}$ such that $\norm{B_{-j}}
\geq \norm{B_{-i}}\rr^{j-i}$, let $\ell_{1}$ be defined as
\[
  \ell_{1} := {\frac{(j-i)\log C_{1} +i\log\norm{B_{-j}}
      -j\log\norm{B_{-i}}}{(i-j)\log\rr + \log\norm{B_{-j}} -
      \log\norm{B_{-i}}}}
\]
and suppose moreover that for all $j+1 \leq k \leq \ell_{1}$ it holds
\begin{equation*}
  \norm{B_{-k}} \leq \norm{B_{-j}}^{\frac{k-i}{j-i}}\norm{B_{-i}}^{-\frac{k-j}{j-i}}.
\end{equation*}
Then $-i$, $-j$ are also consecutive indices in $\newti{\tropx
  F}$.
\end{theorem}
\begin{proof}
  Observe that for $k > 0$ the points $(k, \log \norm{B_{k}})$ lie below the line
  \[
    L_{1}\colon y = \log C_{2}- x\log\RR,
  \]
  due to the fact that $\log{\norm{B_{k}}} \leq C_{2}\RR^{-k}$. In addition,
  let
  \[
    L_{2} \colon y = \log{\norm{B_{j}}} +
    (x-j)\frac{\log{\norm{B_{j}}} - \log{\norm{B_{i}}}}{j-i}
  \]
  be the line containing the segment between $(i, \log{\norm{B_{i}}})$
  and $(j, \log{\norm{B_{j}}})$. Now we want to prove that if this
  segment belongs to $\newt{\tropx F_{d}}$, then it belongs to
  $\newt{\tropx F}$. In order to do that, we have to show that all the
  other points $(k, \log{\norm{B_{k}}})$ lie below $L_{2}$. If
  $j+1 \leq k \leq \ell_{2}$, then this condition is equivalent
  to~\eqref{eq:thm:truncation}. On the other hand, if $k \geq \ell_{2}$,
  then we can prove that $L_{2}$ lies above $L_{1}$, and therefore
  above $(k, \log{\norm{B_{k}}})$. We can do this by pointing out that
  \[
L_{2}(k) -L_{1}(k) \geq L_{2}(\ell_2) - L_{1}(\ell_2) =0.
\]
   Indeed, for the second equality we have
  \[
    \begin{split}
  (j-i)(L_{2}(\ell_2)-L_{1}(\ell_2)) &=  (j-i)\log\norm{B_{j}}  +(\ell_2-j)(\log \norm{B_{j}} -
  \log\norm{B_{i}}) \\
  &+ (i-j)\log C+ \ell_2(j-i) \log\RR\\
  &=(i-j)\log C + (j-\ell_2)\log\norm{B_{i}} \\
  &+(\ell_2-i)\log
   \norm{B_{j}}  +\ell_2(j-i)\log \RR\\
   &= (i-j)\log C  +j\log\norm{B_{i}} -i\log\norm{B_{j}}\\
   &+\ell_2((j-i)\log\RR +\log\norm{B_{j}}-\log\norm{B_{i}}) = 0
  \end{split}
\]
by the definition of $\ell_2$. For the first inequality, it is
sufficient to prove that the slope of $L_{2}$ is larger
than the slope of $L_{1}$. 
This is equivalent to 
\[
  -(j-i)\log \RR \leq \log \norm{B_{j}} - \log\norm{B_{i}}
\]
which is true because we assumed $\norm{B_{j}} \geq
\norm{B_{i}}\RR^{-(j-i)}$. The proof for the left truncation is
identical, providing changing the slope of $L_{1}$ from
$-\log\RR$ to $\log\rr$.
\end{proof}

We now consider a different situation, where we assume to be given the
tropical roots $(\alpha_j)_{j\in T}$ of some tropical Laurent series
for $\tropx f(x)$, and we want to determine the tropical roots of the
modified function $\tropx g(x) := \tropx(f(x) + p(x))$, where
$\tropx p(x)$ is a Laurent polynomial.

An example of this form can be constructed by considering
$\tropx g(x) = \tropx (\eu^{x} + p(x)$); the Newton polygon of the
exponential function can be described explicitly by computing the
convex hull of the nodes $(j, \log b_j)$, for $b_j = \frac{1}{j!}$.
The polynomial $\tropx p(x)$ only modifies a finite number of nodes in
the tropical roots. A natural question is whether it is possible to
determine the new tropical roots
$(\widehat\alpha_j)_{j\in \widehat{T}}$ with a finite number of
comparisons. 


The most natural way to address the problem is to generalize 
the Graham Scan algorithm \cite{gr:72}. To simplify the description, 
we may assume that $p(\lambda) = \gamma$ is a constant. Indeed, 
the case $p(\lambda) = \gamma \lambda^j$ can be dealt with by shifting the 
Newton polygon to the left of $j$ positions, and the general 
case can be seen as a composition of a finite number of 
updates with monomials of this form. 

Our objective is, given $I$ the set of non-zero indices of the 
Laurent series as in 
Definition~\ref{def:I}, and the mapping $j \mapsto b_j$, 
to construct the modified index set $\hat I$ that contains the 
nodes of the updated upper convex hull. As we will prove, 
this new set may be expressed in two ways:
\begin{itemize}
  \item We may have that $I$ and $\hat I$ only differ by a 
    finite number of elements; hence, we can express $\hat I$ 
    by listing such modifications. 
  \item The set $\hat I$ can be obtained from $I$ by dropping 
    all the indices to the right (resp. to the left) of 
    $0$, and adding $0$. 
\end{itemize}
Intuitively, we may proceed 
as follows:
\begin{itemize}
  \item If $\log \gamma$ is below the segment defining the 
    Newton polygon in $0$, we stop, and 
    we do not include $(0, \log \gamma)$ in the nodes.
    Otherwise, we include 
    $0$ as a node in $\widehat{I}$, and move to the next point. 
  \item Starting from $i = 1$,
    we compare the slope of the segment connecting 
    $(0, \log \gamma)$ to $(j_i, \log b_{j_i})$, with 
    $(j_{i+1}, \log b_{j_{i+1}})$. 
    If the latter slope is smaller, or $j_i$ is 
    the last index in $I$, then we stop.       
  \item Otherwise, we remove the index $j_i$ from $I$, 
      and consider the next indices in the previous bullet point. 
\end{itemize}
The above procedure deals with the points 
on the right of $(0, \log \gamma)$, and the same algorithm
must be repeated for the indices on the left. 

As long as $I$ only comprises a finite number of points, 
the previous algorithm is feasible, and terminates within 
$\mathcal O(\# I)$ comparisons. The new tropical polygon 
can be described by a set of indices $\hat I$, obtained with a
finite number of additions and/or removals. However, if the cardinality 
of $I$ is infinite, in general this algorithm may not terminate, and 
could require to perform an infinite number of comparisons. 

\begin{example} \label{ex:upd-1}
  Consider the tropical series associated with the Laurent
  series
  \[
    f(\lambda) = \sum_{j = 1}^\infty \eu^{\frac{j-1}{j}} \lambda^j. 
  \]
  Clearly, the nodes of the Newton polygon are given by 
  $(j, 1 - \frac{1}{j})$; it is easy to verify that all 
  these nodes belong to the upper convex hull, as illustrated in 
  Figure~\ref{fig:ex-upd-1}; the convex hull is represented 
  by the dashed red line. 
  If we modify this function by considering 
  \[
    g(\lambda) = \eu + f(\lambda), 
  \]
  then we need to add the node $(0, 1)$ to the convex hull. 
  If we connect this node to $(j, 1 - \frac{1}{j})$ we obtain a 
  sequence of lines of increasing slope, that converges to 
  the horizontal line $y = 1$. However, the Graham Scan algorithm 
  does not terminate in finite time. The final convex hull 
  correspond to $\{ y \leq 1 \}$, and is depicted by the blue 
  solid line. The lines considered by the Graham Scan algorithm 
  are reported in dotted green. 
\end{example}

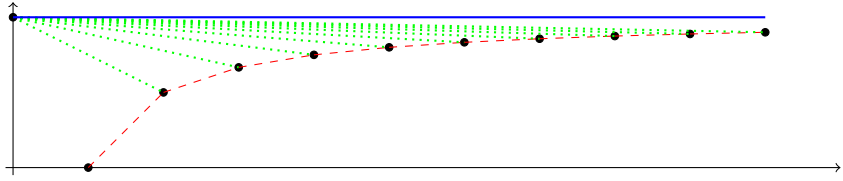
\begin{figure}
  \begin{tikzpicture}
    \draw[thin,->] (0,-.1) -- (0,2.2);
    \draw[thin,->] (-.1,0) -- (11,0);
    \filldraw (1, 0) circle (0.05);
    \foreach \j in {2, ..., 10} {
      \filldraw (\j, 2 - 2/\j) circle (0.05);
      \draw[red,dashed] (\j-1,{2-2/(\j-1)}) -- (\j, 2 - 2/\j);
      \draw[thick, green,dotted] (0,2) --  (\j, 2 - 2/\j);
    }
    \filldraw (0, 2) circle (0.05);
    \draw[blue, thick] (0,2) -- (10, 2);
  \end{tikzpicture}
  \caption{Lines considered by the Graham Scan algorithm in
     Example~\ref{ex:upd-1}. The algorithm does not terminate 
    within a finite number of slope comparisons.}
  \label{fig:ex-upd-1}
\end{figure}

We remark that, for this example, the final set $\hat I$ 
is easy to describe: it is composed only by the index $0$, 
which corresponds to a tropical root of infinite multiplicity. 

It turns out that the case described in the example is 
essentially the only one that can ``go wrong''. If we detect 
this situation early, we may adopt the standard Graham Scan procedure 
for all other cases. This is 
guaranteed by the following result. 

\begin{theorem} \label{thm:gs}
  Let $(j, \log b_j)$ be the nodes of a Newton polygon, 
  $j \in I \subseteq \mathbb Z$; similarly, 
  let $\hat I$ be the set of indices corresponding to 
  vertices of the Newton polygon for 
  \[
    \hat b_{j} = \begin{cases}
      \gamma & \text{if } {j = 0} \\
      b_j & \text{otherwise}
    \end{cases}.
  \]
  Then, if any of the following 
  conditions hold, the Graham Scan algorithm applied 
  to the right of $0$ terminates 
  within a finite number of steps. 
  \begin{enumerate}
    \item[(i)] $\alpha_\infty = \infty$
    \item[(ii)] $\alpha_\infty < \infty$ and 
      $
        \limsup_{j \in I} \left[ 
          \log(b_j) - j \log \alpha_\infty
        \right] > 
        \log\gamma. 
      $
  \end{enumerate}
  If instead $\alpha_\infty < \infty$, but 
  condition $(ii)$ is not satisfied, then $\hat I \cap \{ j \geq 0 \} = \{ 0 \}$, 
  and the root at zero has infinite multiplicity. 
\end{theorem}

\begin{proof}
  Let us assume that $\alpha_\infty = \infty$. If there is a finite 
  number of tropical roots, the result is clearly true. Otherwise, 
  $\alpha_\infty = \infty$ implies that the tropical roots are 
  unbounded. In particular, 
  the Graham Scan algorithm proceeds comparing the slopes defined by 
  the sequences:
  \[
    s_i := 
      \frac{\log b_{j_i} - \log \gamma}{j_i}, \qquad 
    t_i := \frac{\log b_{j_{i+1}} - \log b_{j}}{j_{i+1} - j_i},
  \]
  and stops as soon as $s_i \geq t_i$. We note that, as long 
  as this condition does not hold, the sequence $s_i$ is 
  increasing, since the ``next'' node $(j_{i+1}, \log b_{j_{i+1}})$
  is above the line passing 
  through $(0, \log \gamma)$ and $(j_i, \log b_{j_i})$. 
  This situation is visible, for instance, 
  in Figure~\ref{fig:ex-upd-1}, where none of the lines satisfy 
  the condition, and hence the slopes keep increasing. 
  In contrast, 
  the sequence $t_i$ converges to $-\infty$. Hence, there exists a 
  finite index $i$ where $s_i \geq t_i$, and 
  where the condition is satisfied. 

  It remains to consider the case $\alpha_\infty < \infty$. 
  In this setting, it is not restrictive to assume that 
  $\log \alpha_\infty = 0$, and therefore $\alpha_\infty = 1$: 
  we just need to modify the function $\tropx f(x)$
  by scaling the variable as $\tropx f(\alpha_\infty^{-1} x)$, 
  which yields the Laurent coefficients $b_j \alpha_\infty^{-j}$. 

  With this choice, the slope of the segments in the tropical roots 
  converges to $0$, and the plot is ``asymptotically flat''. 
  In addition, all the slopes are non-negative, and therefore 
  the sequence of $\log b_j$ are non-decreasing. 
  We now define $\xi := \lim_{j \to \infty} \log b_j$, 
  and we distinguish 
  two cases, one where $\xi > \log \gamma$, and the other 
  where $\xi \leq \log \gamma$. 

  If $\xi >\log \gamma$, then there must be at least one point 
  $(j, \log b_j)$ such that the slope of the segment connecting 
  $(0, \log \gamma)$ to it is strictly positive, and this 
  also holds for all the following points, since the 
  $b_j$ are non-decreasing. We can now use the same argument 
  as before: the sequences of slopes $s_i$ and $t_i$ are 
  such that $s_i$ is non-decreasing as long as the 
  stopping condition is not satisfied, and $t_i \to 0$. Hence, 
  the algorithm terminates in a finite number of steps. 

  Otherwise, if $\xi \leq \log \gamma$, the horizontal line starting
  from $(0, \log \gamma)$ is above all other points, but any other
  line passing through the same point and with negative slope
  necessarily intersects the previous Newton polygon. Hence, $0$ is a
  tropical root of infinite multiplicity, and all nodes with positive
  indices need to be removed. The statement follows by rephrasing the
  claim on the original $b_j$, for a generic
  $\log \alpha_\infty \neq 0$.
\end{proof}

\section{Applications}
\label{sec:applications}

In this section, we discuss a few applications of our theory. Throughout 
this section, the radii of inclusion obtained numerically are only reported 
with a couple of significant digits, for improved readability. 

\subsection{Updating the tropical roots}
  \label{sec:updatingRoots}
  
We begin by 
considering the problem of finding approximate inclusion sets
for the eigenvalues of a Laurent series (or 
its zeros, for scalar problems). We assume that 
a known (scalar, for simplicity) Taylor or Laurent series 
is perturbed in a finite number of coefficients, and we leverage
the results in Theorem~\ref{thm:gs}. 


  \begin{example} \label{ex:example1_np}
  
  We consider
  \[
    g(\lambda) = \eu^\lambda + p(\lambda), \qquad 
    p(\lambda) := 12\lambda - \frac{\lambda^2}{5} + 12\lambda^3 - 0.04 \lambda^4 + 10^{-3} \lambda^5 
     - 0.002 \lambda^6. 
  \]
  The Newton polygon for $\eu^\lambda$ is composed of all the nodes 
  $(i, - \log i!)$, for $i \in \mathbb Z$. Hence, we may apply 
  Theorem~\ref{thm:gs} to compute the Newton polygon for 
  $g(\lambda)$, and then rely on Theorem~\ref{thm:evsLoc} to 
  construct inclusion results. The computation of the convex 
  hull yields 3 segments with well-separated slopes 
  according to Theorem~\ref{thm:evsLoc}, 
  as visible in Figure~\ref{fig:ex-upd-1}. The 
  exponential of their negative slopes yields 
  (approximately) the tropical roots $\alpha_1 \approx  0.0769$,
  $\alpha_2 \approx 1.0337$, and $\alpha_3 \approx 12.2446$. 
  Theorem~\ref{thm:evsLoc} and Theorem~\ref{thm:lower-upper}
  yield the following 
  inclusions:
  \begin{itemize}
    \item The open disc of radius $r \approx 0.038$ and 
      centered at zero, 
      corresponding to the slope of the first segment of 
      the Newton polygon, does not contain any root, 
      as predicted by Theorem~\ref{thm:lower-upper}. 
    \item Similarly, the annulus $\annop(0.17, 0.46)$ does not 
      contain any root; 
      in addition, there is exactly one root in the annulus
      $\ann(0.038, 0.174)$.
    \item Finally, the annulus $\annop(2.40, 5.27)$ does not contain 
      roots, and there are exactly two roots in the annulus 
      $\ann(0.46, 2.40)$.
  \end{itemize}
  These inclusions are displayed in the right 
  plot of Figure~\ref{fig:example1}. The roots in that Figure 
  have been approximated by running a few steps of Newton's 
  iteration starting from a fine grid. 

  \begin{figure}
    \begin{minipage}{.5\linewidth}
      \includegraphics[width=\linewidth]{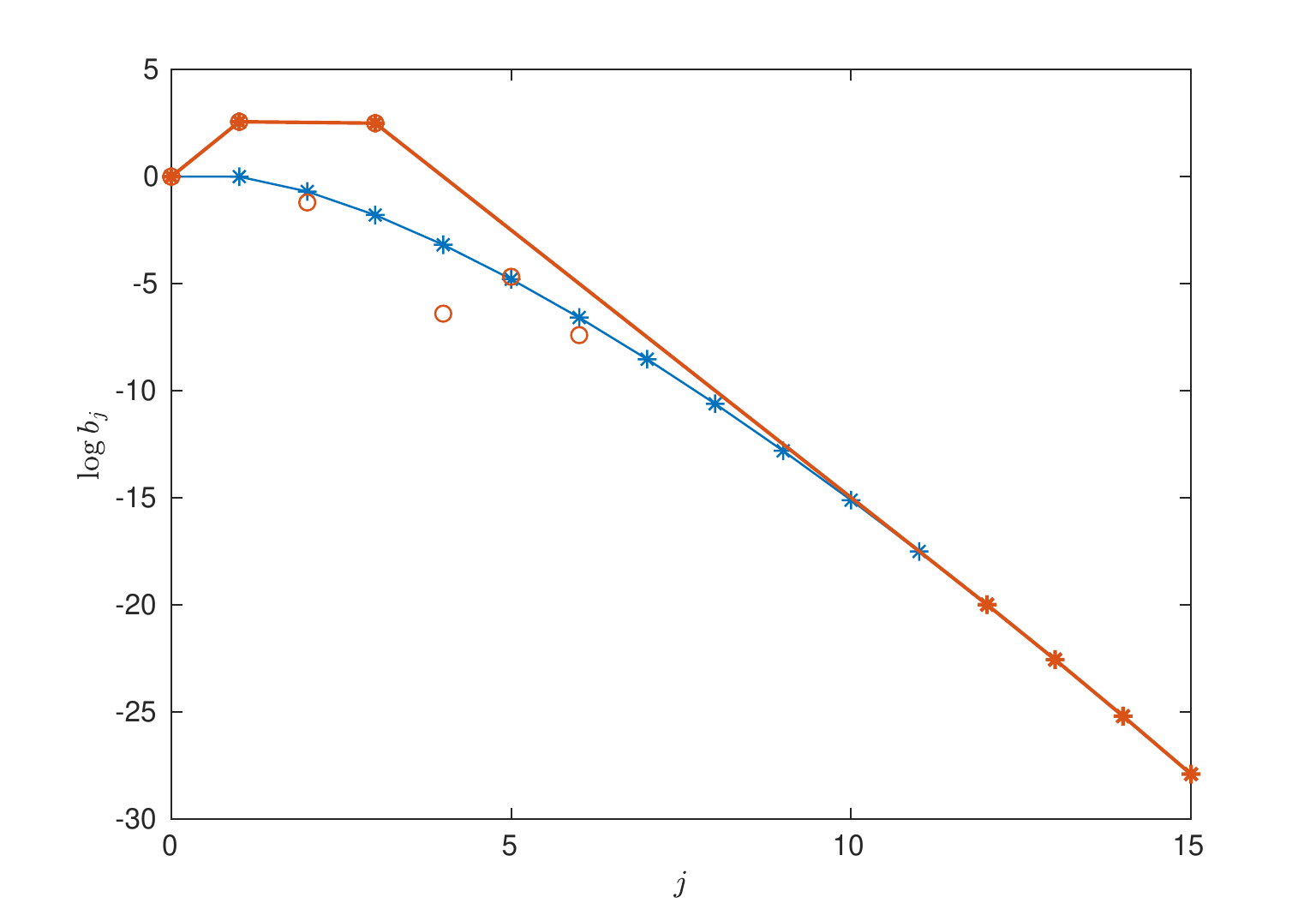}
    \end{minipage}~\begin{minipage}{.5\linewidth}
      \includegraphics[width=\linewidth]{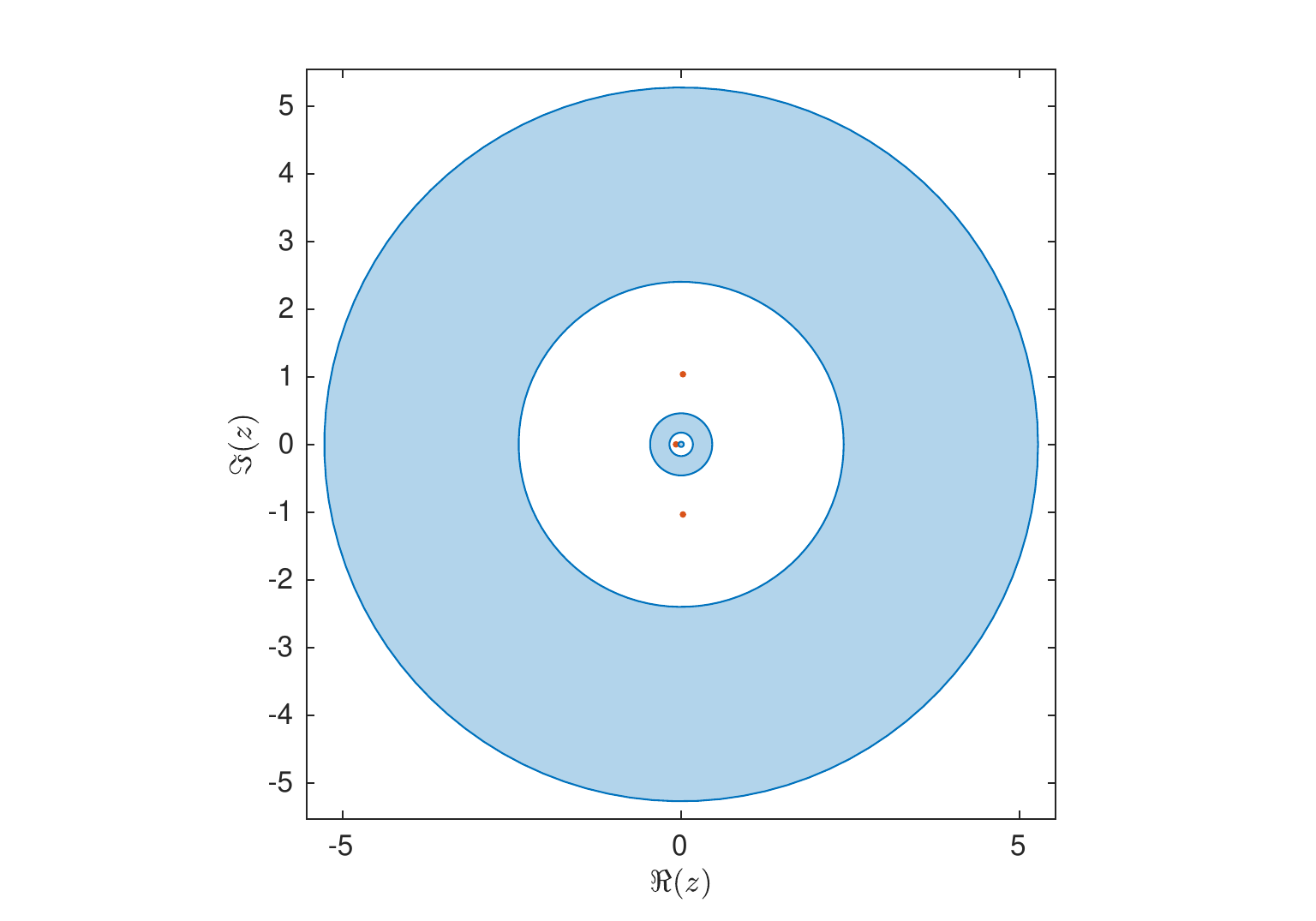}
    \end{minipage}
    \caption{On the left, the Newton polygon of 
    $\eu^\lambda $ (in blue) and the one 
    associated with $g(\lambda)$ (in red). On the right, the 
    exclusion annuli obtained from Theorem~\ref{thm:evsLoc}
    and the approximated roots. }
    \label{fig:example1}
  \end{figure}

\end{example}

  We now consider a variation of Example~\ref{ex:example1_np}
  where we deal with a ``proper" meromorphic function, i.e., not just a Taylor series. 
  \begin{example} \label{ex:example2_np}
    Let 
    \[
      f(\lambda) = \eu^\lambda + \eu^{\frac 1\lambda} - 1,
    \]
and let $f_{n}(\lambda) = \sum_{j=-n}^{n}b_{j}\lambda^{j}$ be a
truncation of $f(\lambda)$ expressed as a Laurent series. Consider the function
    \[
            g(\lambda) = f_{n}(\lambda) + p(\lambda)
      \]
    where $p(\lambda)$ is the Laurent polynomial defined by 
    \[
      p(\lambda) := \eu^6 \lambda^{-9} + \eu^{12} \lambda^{-3} + \eu + \eu^2 \lambda^2 +
      \eu^{-10} \lambda^4 + \eu^{-14} \lambda^5 + \eu^{-20} \lambda^5.  
    \]
    Similarly to Example~\ref{ex:example1_np}, the Newton Polygon for
     $f_{n}(\lambda)$ is easily determined, and is
    composed by the nodes of coordinates $(i, -\log(\abs{i}!))$, with
    $\abs{i} \leq n$. Here we set $n=45$ and  computed numerically the updated
    Newton polygon using Theorem~\ref{thm:gs}. This yields a Newton
    polygon with fewer nodes, with the following inclusion/exclusion
    annuli:
    \begin{itemize}
      \item The annulus $\mathcal A_1 := \annop(0.05, 0.17)$ does not contain 
        any root, and the disc inside it contains $34$ roots minus 
        poles. 
      \item The annulus $\mathcal A_2 = \annop(0.79, 3.41)$ does not contain 
        any root, and exactly $6$ roots are contained between 
        $\mathcal A_1$ and $\mathcal A_2$. 
    \end{itemize}
    \begin{figure}
      \begin{minipage}{.5\linewidth}
        \includegraphics[width=\linewidth]{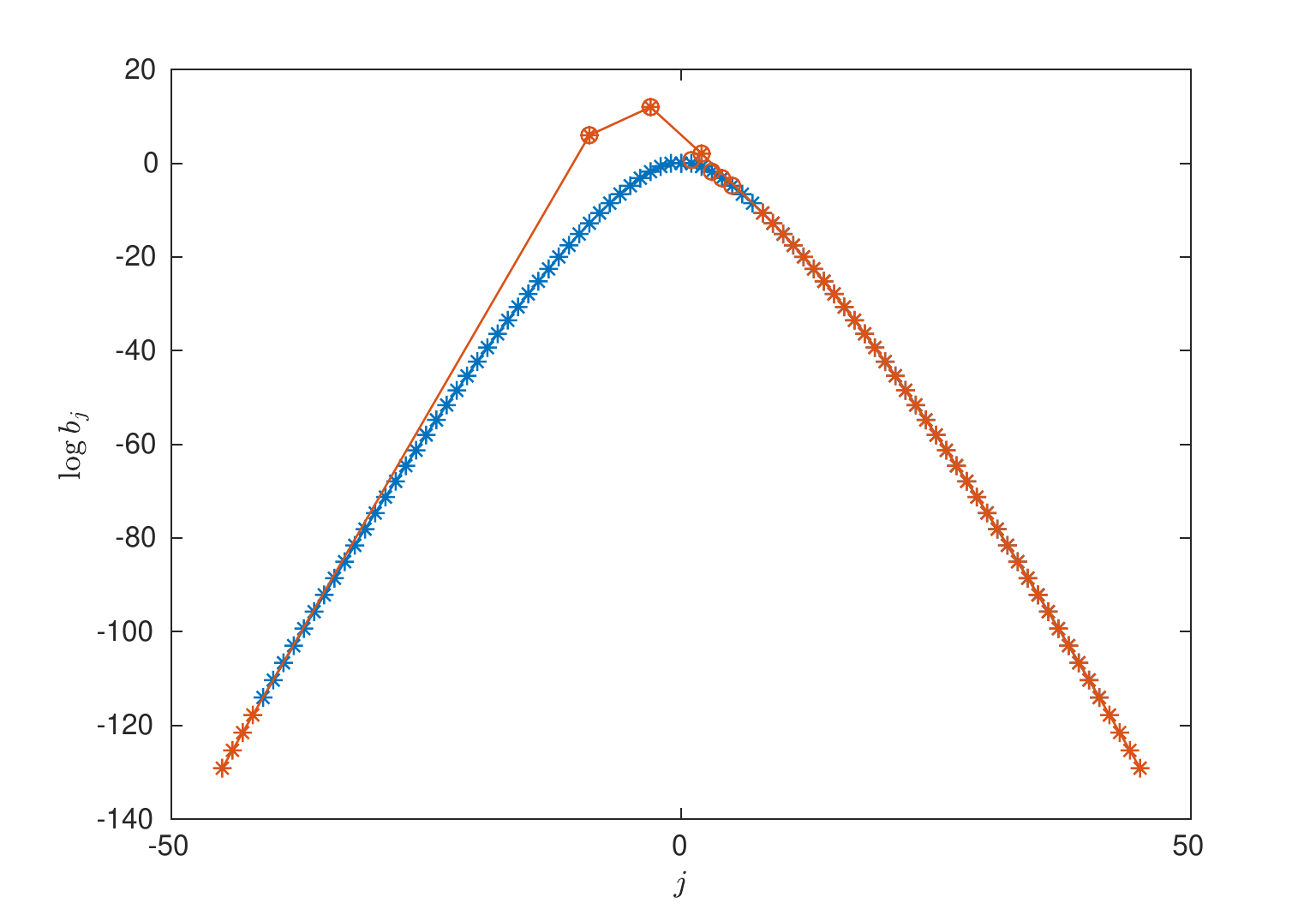}
      \end{minipage}~\begin{minipage}{.5\linewidth}
        \includegraphics[width=\linewidth]{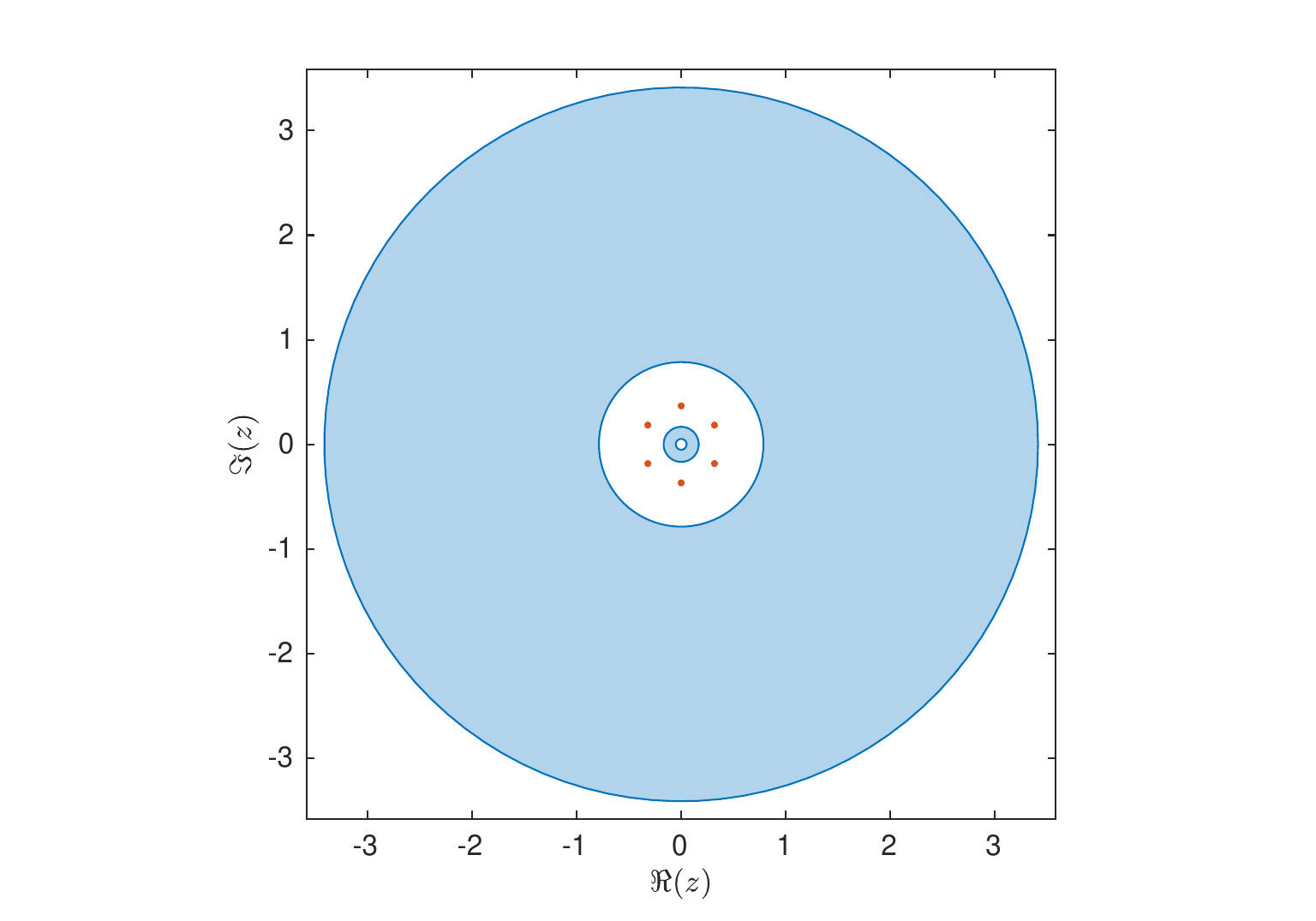}
      \end{minipage}
      \caption{On the left, the Newton polygon of 
        $f_{45}(\lambda)$ (in blue) and the one 
        associated with $g(\lambda)$ (in red). On the right, a zoom on the 
        exclusion annuli obtained from Theorem~\ref{thm:evsLoc}
        and the approximated roots. }
      \label{fig:example2}
    \end{figure}
  \end{example}

  \subsection{Estimating the number of nodes in contour integrals}

  As we hinted  in the introduction, tropical roots are useful if we
  desire to solve an eigenvalue problem through the means of a contour
  integral algorithm, because they give us the location of the target
  set $\Omega$ where to look. One of the most important parameters of
  these procedures is the quadrature rule and the number of quadrature
  points $N$ to be used. The choice often falls on the trapezoidal rule,
  because it is easy to implement and has an exponential drop of
  the error on elliptical contours. However, as far as we know and as this paper is being written, there is
  currently no way to determine an ``optimal'' $N$ a priori so that the
  backward error of the eigenpairs is below a given threshold.
  
  In this section we show how the results of Theorem~\ref{thm:evsLoc}
  can be used to justify a particular choice of $N$. In 2016 Van Barel
  and Kravanja showed a relationship between contour integrals of
  holomorphic functions and filter functions~\cite{bakr:16}. More
  specifically, let $\Omega = \disk(r)$ be the target set and
  \[
    A_{p} = \frac{1}{2\pi\iu r}\int_{\partial \Omega}z^{p}F(z)^{-1}\, dz
    \approx \frac{1}{2\pi\iu r}\sum_{j=0}^{N-1}w_{j}F(z_{j})^{-1} =:\wt{A}_{0}
  \]
  be the basic contour integral for Beyn's algorithm and its
  approximation with the quadrature rule of points and weights
  $(z_{j},w_{j})$~\cite{be:12}. They showed that when we use the
  trapezoidal rule, we can associate the filter function
  
  \[b_{p}(z) = \frac{z^{p}}{1-(\frac{z}{r})^{N}}
  \]
  to $A_{p}$, where $b_{0}(z)$ approximates the function
  \[
    \chi_{{}_\Omega}(z) = \left\{
      \begin{array}{ll}
        1& \text{if } z\in\Omega,\\
        0&  \text{if } z \not\in \Omega.
      \end{array}
    \right .
    \]
  Note  that for  $\epsilon \ll 1$, it follows $\abs{b_{0}(z)} < \epsilon$
    approximately when $\abs{z} > \epsilon^{-1/N}$. In general they
    proved that any eigenvalue $\lambda$ of $F(z)$ that lies outside $\Omega$
    becomes an element of noise for the approximation of $A_p$ of order
    $\bigO(\abs{b_0(\lambda)})$. This has two immediate consequences: if
    there are eigenvalues $\lambda$ near $\Omega$, then the
    quadrature rule needs a large number $N$ of quadrature points to
    reduce $\abs{b_0(z)}$ and hence the noise; on the other hand, if
    there are no eigenvalues near $\Omega$, then there is no need of a large
    $N$, with an obvious saving of computational time. Therefore
    Theorem~\ref{thm:evsLoc} is a perfect tool for this task: the
    annulus of exclusion can help us set an ideal parameter $N$. The next
    examples will serve as a clarification.

    \begin{example}
      \label{ex:polyRoots}
      Consider a matrix polynomial $P(\lambda) = \sum_{j=0}^{4}B_{j}\lambda^{j}$ generated with the \textsc{Matlab}
      commands
  \begin{verbatim}
     rng(42); n = 20;
     B0 = randn(n); B1 = 1e5*randn(n);
     B2 = randn(n); B3 = 1e-2*randn(n); B4 = 1e3*randn(n);
     \end{verbatim}
      The associated tropical polynomial $\tropx P(x) =  \max_{1\leq j\leq
        4}\normt{B_{j}}x^{j}$ has two roots
  \[
  \alpha_{1}  = \normt{B_{0}}/\normt{B_{1}}\approx
  10^{-5},\qquad
  \alpha_{2} = (\normt{B_{1}}/\normt{B_{4}})^{1/3} \approx 4.8
  \]
  of multiplicity $1$ and $3$, respectively. It holds that
  $\delta_1 < (1+2\cond(B_0))^{-2} \approx 10^{-5}$, hence there are
  $20$ eigenvalues in $\disk((1+2\cond(B_{0}))\alpha_{1})$, and no
  eigenvalues in
  $\annop((1+2\cond(B_{0}))\alpha_{1}, (1+2\cond(B_{0}))^{-1}\alpha_{2})
  \approx\annop(0.0012, 0.039)$. If we are interested in the eigenvalues
  inside $\Omega = \disk((1+2\cond(B_{0}))\alpha_{1}):=\disk(r)$, then
  we can set $N=10$, given that
      \[
          b_{0}(z) = \frac{1}{1-(\frac{z}{r})^{12}} \approx \bigO(10^{-16})
    \]
      when $\abs{z}\approx 0.039$. Define now the backward error of
      an eigenpair $(\lambda, v)$ as
      \[
        \eta(\lambda,v) = \frac{\norm{F(\lambda)v}_{2}}{\norm{F}_{{}_{\Omega}}\norm{v}_{2}},
      \]
      with
      \[
        \norm{F}_{{}_{\Omega}} : = \sup_{z\in \Omega}\norm{F(z)}_2.
        \]
      Then, a  basic implementation of
      Beyn's algorithm returns a backward error of approximately
      $10^{-15}$ for all the $20$ eigenpairs.

    \end{example}
  
    \begin{example}
      Consider again the polynomial $P(\lambda)$ of
  Example~\ref{ex:polyRoots} and the function $F(\lambda) = \lambda I
  + \eu^{-\lambda}B$, where $B\in \R^{20\times 20}$ is a randomly generated matrix
  and $I$ is the identity of the same size. Using the tools developed in this section, we
  want to find the tropical roots of $\tropx(P(x) + F(x)) =: \tropx
  G(x)$. Note that  $\tropx G(x) = \sup_{j\in \N}\normt{C_{j}}x^{j}$ with
  \[
    \left\{
    \begin{split}
      C_{j} &= B_{j} +\frac{C}{j!}, \qquad \text{for }j=0,2,3,4,\\
      C_{1} &= B_{1} + C + I,\\
      C_{j} &= \frac{C}{j!}, \qquad \text{for }j > 4.
    \end{split}\right.
  \]
  In Figure~\ref{fig:expPoly}
   we plotted the Newton polygon for $\tropx G(x)$ (using the spectral norm). There we can see the
   first three tropical roots
 \[
     \begin{split}
  \alpha_{1}  &= \normt{C_{0}}/\normt{C_{1}}\approx
  2\cdot10^{-5}\\
  \alpha_{2} &= (\normt{C_{1}}/\normt{C_{4}})^{1/3} \approx 4.8\\
  \alpha_{3} &= (\normt{C_{4}}/\normt{C_{21}})^{1/17} \approx 21.3
  \end{split}
  \]
  By applying the same reasoning of Example~\ref{ex:polyRoots}, we know
  there still are $20$ eigenvalues in
  $\Omega = \disk((1+2\cond(C_{0}))\alpha_{1}):=\disk(r)$ and no
  eigenvalues in the annulus
  $\annop((1+2\cond(C_{0}))\alpha_{1}, (1+2\cond(C_{0}))^{-1}\alpha_{2})
  \approx\annop(0.004, 0.025)$. Hence we can set $N=18$ so that
  $b_{0}(z) = \bigO(10^{-15})$ for $\abs{z} \approx 0.025$. Once again,
  a basic implementation of Beyn's algorithm under these settings return
  all $20$ eigenvalues with a backward error of order
  $\bigO(10^{-15})$.
  
  \begin{figure}
      \centering\includegraphics[width=0.85\textwidth]{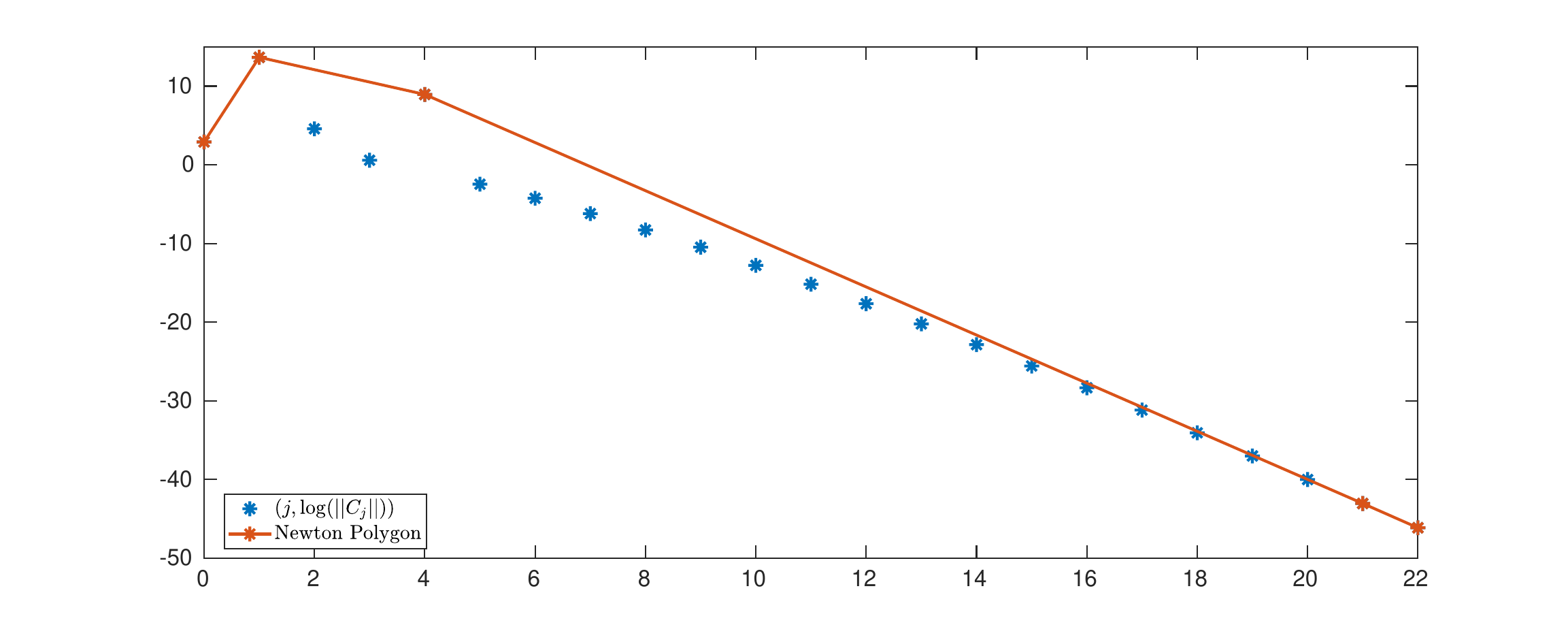}
      \caption{The (truncated) Newton polygon of $\tropx G(x)$. We can
        clearly see the first $3$ tropical roots.}
      \label{fig:expPoly}
  \end{figure}

    \end{example}

  \section{Conclusion}
\label{sec:conclusions}
In this work we generalized the concept of tropical roots from
polynomials to functions expressed as Taylor or Laurent series. We
showed that in order to have a consistent theory, the tropical roots
of $\tropx f(x)$ include, possibly, not only the points of nondifferentiability but also 
the extrema of the interval where $\tropx f(x)$ is defined as a
real-valued function (and zero).

In the second part we proved that the eigenvalues localization theorem
stated for matrix polynomials in~\cite{nosh:15} holds true for
meromorphic Laurent functions, under minor adjustments. Furthermore,
we showed how to update the Newton polygon of a Laurent function
$\tropx f(x)$ to obtain the Newton polygon of $\tropx (f(x) +p(x))$,
where $p(x)$ is a polynomial, and we proved that the strategy
terminates almost surely. Finally, we hinted at a new possible
application, where the presence or absence of eigenvalues in the
neighbourhood of a target set $\tset$ can help determine the optimal
number of quadrature points for contour integral eigensolvers.


Tropical roots of polynomials are already beneficial for scaling and localization in the context of polynomial eigenvalue problems and scalar polynomial equation solvers. We hope that this work may initiate the use of tropical roots of meromorphic functions for similar purposes.

 \printbibliography  

\end{document}